\documentclass[11pt]{article}

\usepackage{mathpazo}   
\usepackage{comment}
\usepackage[title]{appendix}
\usepackage{wrapfig}
\usepackage{subcaption}
\usepackage{tikz}
\usepackage{tikzscale}
\usepackage{pgfplots}

\usepackage{pgfplots}
\pgfplotsset{compat=newest}
\usepgfplotslibrary{groupplots}
\usepgfplotslibrary{patchplots}
\usepgfplotslibrary{fillbetween}
\usetikzlibrary{arrows.meta}
\usetikzlibrary{backgrounds}

\usepackage[font=footnotesize]{caption}
\usepackage{nameref}
\usepackage{mathtools}
\usepackage{empheq}
\usepackage{comment}
\usepackage[shortlabels,inline]{enumitem}
\setlist[enumerate]{nosep}
\usepackage[colorlinks=true,
linkcolor=refkey,
urlcolor=lblue,
citecolor=red]{hyperref}
\usepackage[doc,wmm,hhb]{optional}
\usepackage{xcolor}

\usepackage{float}
\usepackage{soul}
\usepackage{graphicx}
\definecolor{labelkey}{rgb}{0,0.08,0.45}
\definecolor{refkey}{rgb}{0,0.6,0.0}
\definecolor{Brown}{rgb}{0.45,0.0,0.05}
\definecolor{lime}{rgb}{0.00,0.8,0.0}
\definecolor{lblue}{rgb}{0.5,0.5,0.99}
\definecolor{OliveGreen}{rgb}{0,0.6,0}
\definecolor{tyrianpurple}{rgb}{0.4, 0.01, 0.24}

\colorlet{hlcyan}{cyan!30}

\usepackage{stmaryrd}
\usepackage{amssymb}

\makeatletter
\def\namedlabel#1#2{\begingroup
   \def\@currentlabel{#2}%
   \label{#1}\endgroup
}
\makeatother

\newcommand{\seppthree}{\setlength{\itemsep}{-3pt}}

\usepackage[margin=0.92in,footskip=0.25in]{geometry}
\newcommand{\weakly}{\ensuremath{\:{\rightharpoonup}\:}}

\newcommand{\nnn}{\ensuremath{{n\in{\mathbb N}}}}
\newcommand{\kkk}{\ensuremath{{k\in{\mathbb N}}}}
\newcommand{\thalb}{\ensuremath{\tfrac{1}{2}}}
\newcommand{\menge}[2]{\big\{{#1}~\big |~{#2}\big\}}

\newcommand{\fenv}[1]%
{\ensuremath{\,\overrightarrow{\operatorname{env}}_{#1}}}
\newcommand{\benv}[1]%
{\ensuremath{\,\overleftarrow{\operatorname{env}}_{#1}}}

\newcommand{\RR}{\ensuremath{\mathbb R}}

\newcommand{\ZZ}{\ensuremath{\mathbb Z}}
\newcommand{\NN}{\ensuremath{\mathbb N}}

\newcommand{\Fix}{\ensuremath{\operatorname{Fix}}}

\newcommand{\pinf}{\ensuremath{+\infty}}

\newcommand{\Bf}{\mathrm{B}}

\providecommand{\BN}{Bo\c{t}--Nguyen}
{\begin{list}{}{%
\settowidth{\labelwidth}{\textrm{#1~}}%
\setlength{\leftmargin}{\labelwidth+\labelsep}}}
{\end{list}}
\usepackage{amsthm}
\usepackage{aliascnt}
\makeatletter
\def\th@plain{%
	\thm@notefont{}
	\itshape 
}
\def\th@definition{%
	\thm@notefont{}
	\normalfont 
}
\makeatother
\usepackage[capitalize,nameinlink]{cleveref}
\crefname{equation}{}{equations}

\crefname{item}{}{items}
\crefname{enumi}{}{}
\newtheorem{theorem}{Theorem}[section]

\newaliascnt{lemma}{theorem}
\newtheorem{lemma}[lemma]{Lemma}
\aliascntresetthe{lemma}
\crefname{lemma}{Lemma}{Lemmas}
\Crefname{lemma}{Lemma}{Lemmas}

\newaliascnt{lem}{theorem}

\aliascntresetthe{lem}
\crefname{lem}{Lemma}{Lemmas}
\Crefname{lem}{Lemma}{Lemmas}

\newaliascnt{corollary}{theorem}
\newtheorem{corollary}[corollary]{Corollary}
\aliascntresetthe{corollary}
\crefname{corollary}{Corollary}{Corollaries}
\Crefname{corollary}{Corollary}{Corollaries}

\newaliascnt{cor}{theorem}

\aliascntresetthe{cor}
\crefname{cor}{Corollary}{Corollaries}
\Crefname{cor}{Corollary}{Corollaries}

\newaliascnt{proposition}{theorem}
\newtheorem{proposition}[proposition]{Proposition}
\aliascntresetthe{proposition}
\crefname{proposition}{Proposition}{Propositions}
\Crefname{proposition}{Proposition}{Propositions}

\newaliascnt{prop}{theorem}

\aliascntresetthe{prop}
\crefname{prop}{Proposition}{Propositions}
\Crefname{prop}{Proposition}{Propositions}

\newaliascnt{definition}{theorem}
\newtheorem{definition}[definition]{Definition}
\aliascntresetthe{definition}
\crefname{definition}{Definition}{Definitions}
\Crefname{definition}{Definition}{Definitions}

\newaliascnt{defn}{theorem}

\aliascntresetthe{defn}
\crefname{defn}{Definition}{Definitions}
\Crefname{defn}{Definition}{Definitions}

\newaliascnt{thm}{theorem}

\aliascntresetthe{thm}
\crefname{thm}{Theorem}{Theorems}
\Crefname{thm}{Theorem}{Theorems}

\newaliascnt{example}{theorem}
\newtheorem{example}[example]{Example}
\aliascntresetthe{example}
\crefname{example}{Example}{Examples}
\Crefname{example}{Example}{Examples}

\newaliascnt{ex}{theorem}
\newtheorem{ex}[ex]{Example}
\aliascntresetthe{ex}
\crefname{ex}{Example}{Examples}
\Crefname{ex}{Example}{Examples}

\newaliascnt{fact}{theorem}
\newtheorem{fact}[fact]{Fact}
\aliascntresetthe{fact}
\crefname{fact}{Fact}{Facts}
\Crefname{fact}{Fact}{Facts}

\newaliascnt{conjecture}{theorem}
\newtheorem{conjecture}[conjecture]{Conjecture}
\aliascntresetthe{conjecture}
\crefname{conjecture}{Conjecture}{Conjectures}
\Crefname{conjecture}{Conjecture}{Conjectures}

\newaliascnt{remark}{theorem}
\newtheorem{remark}[remark]{Remark}
\aliascntresetthe{remark}
\crefname{remark}{Remark}{Remarks}
\Crefname{remark}{Remark}{Remarks}

\newaliascnt{rem}{theorem}

\aliascntresetthe{rem}
\crefname{rem}{Remark}{Remarks}
\Crefname{rem}{Remark}{Remarks}

\crefname{chapter}{Appendix}{chapters}

\providecommand{\norm}[1]{\lVert#1\rVert}

\providecommand{\innp}[1]{\langle#1\rangle}

\providecommand{\RR}{\mathbb{R}}

\providecommand{\ZZ}{{\bf Z}}

\providecommand{\NN}{\mathbb{N}}

\providecommand{\RR}{\mathbb{R}}
\providecommand{\NN}{\mathbb{N}}

\definecolor{myblue}{rgb}{0.9,0.9,0.98}

  \newcommand*\mybluebox[1]{%
    \colorbox{myblue}{\hspace{1em}#1\hspace{1em}}}

\allowdisplaybreaks 

\usepackage{relsize}

\begin{document}

\setlength{\abovedisplayskip}{8pt}
\setlength{\belowdisplayskip}{8pt}	
\newsavebox\myboxA
\newsavebox\myboxB
\newlength\mylenA

\newcommand*\xoverline[2][0.75]{%
    \sbox{\myboxA}{$#2$}%
    \setbox\myboxB\null
    \ht\myboxB=\ht\myboxA%
    \dp\myboxB=\dp\myboxA%
    \wd\myboxB=#1\wd\myboxA
    \sbox\myboxB{$\overline{\copy\myboxB}$}
    \setlength\mylenA{\the\wd\myboxA}
    \addtolength\mylenA{-\the\wd\myboxB}%
    \ifdim\wd\myboxB<\wd\myboxA%
       \rlap{\hskip 0.5\mylenA\usebox\myboxB}{\usebox\myboxA}%
    \else
        \hskip -0.5\mylenA\rlap{\usebox\myboxA}{\hskip 0.5\mylenA\usebox\myboxB}%
    \fi}
\makeatother

\makeatletter
\renewcommand*\env@matrix[1][\arraystretch]{%
  \edef\arraystretch{#1}%
  \hskip -\arraycolsep
  \let\@ifnextchar\new@ifnextchar
  \array{*\c@MaxMatrixCols c}}
\makeatother

\providecommand{\wbar}{\xoverline[0.9]{w}}
\providecommand{\ubar}{\xoverline{u}}

\newcommand{\nn}[1]{\ensuremath{\textstyle\mathsmaller{({#1})}}}
\newcommand{\crefpart}[2]{%
  \hyperref[#2]{\namecref{#1}~\labelcref*{#1}~\ref*{#2}}%
}
\newcommand\bigzero{\makebox(0,0){\text{\LARGE0}}}
	

%

\author{
Heinz H.\ Bauschke\thanks{
Mathematics, University
of British Columbia,
Kelowna, B.C.\ V1V~1V7, Canada. E-mail:
\texttt{heinz.bauschke@ubc.ca}.}~~and
Yuan Gao\thanks{
Mathematics, University
of British Columbia,
Kelowna, B.C.\ V1V~1V7, Canada. 
 E-mail: \texttt{y.gao@ubc.ca}.}
}

\title{\textsf{ 
Bo\c{t}–Nguyen Acceleration, Weighted Mean Ergodic Iteration,\\ 
and the Beta-Binomial Distribution
}
}

\date{July 28, 2026}

\maketitle

\begin{abstract}
In 2023, Bo\c{t} and Nguyen introduced a new class of accelerated algorithms 
for finding a fixed point of a nonexpansive operator as the weak limit of
a sequence. In this paper, we analyze a particular instance of their algorithm 
when the nonexpansive operator is assumed to be linear. 
Surprisingly, the Bo\c{t}--Nguyen acceleration then fits naturally into the 
framework of weighted mean ergodic iterations. 
This allows us to identify the weak limit as the projection of the starting point 
onto the fixed point set. Moreover, the weights involved are closely related 
to the beta-binomial distribution. Finally, when the parameter is equal to $4$, 
then we obtain strong convergence of the iterates.

\end{abstract}
{ 

}

    \section{Introduction}
Throughout this paper, we assume that 
\begin{empheq}[box=\mybluebox]{equation}
\text{$X$ is a real Hilbert space,}
\end{empheq}
with inner product $\innp{\cdot,\cdot}$ and induced norm $\norm{\cdot}$.
The problem of finding a fixed point of a nonexpansive operator acting on $X$
is of fundamental importance in nonlinear analysis and optimization; 
see, e.g., \cite{BC2017}. 
Recently, Bo\c{t} and Nguyen \cite{BN} introduced a new class of optimization algorithm 
for finding a fixed point, based on the idea of enhancing the 
classical Krasnoselskii-Mann iteration with Nesterov momentum. 
In its most basic form (namely 
\cite[Algorithm~1 with $s=1$ and $x_0=x_1$]{BN}), their algorithm has the 
following pleasing form\footnote{See \cite[(12)]{BN}. Their general version 
\cite[(11)]{BN} features 
another parameter, another difference of vectors, and allows for $x_1\neq x_0$.} 
with remarkable convergence properties:

\begin{fact}[{\BN\ acceleration}] 
{\rm\cite[Remark~2.1, and Theorems~3.4 and 3.5]{BN}}
\label{BNaccel}
Let $T\colon X\to X$ be nonexpansive with $\Fix T\neq\varnothing$, 
let $x_0\in X$, and let $\alpha>2$. 
Set $x_1 := x_0$ and, for $n\geq 1$, generate the next iterate via 
\begin{equation}
x_{n+1} := \Big( 1-\frac{\alpha}{2(n+\alpha)}\Big)x_n
+\frac{\alpha}{2(n+\alpha)}Tx_n+\frac{n}{n+\alpha}(Tx_n-Tx_{n-1}).
\end{equation}
Then the sequence $(x_n)_\nnn$ converges weakly to a fixed point of $T$. 
Moreover, $\|x_n-x_{n+1}\|=o(1/n)$ and $\|x_n-Tx_n\|=o(1/n)$.
\end{fact}

Motivated by the good numerical performance of the \BN\ acceleration 
(see \cite[Section~5]{BN}), we set out to understand the \BN\ iteration better 
when $T$ is also assumed to be \textbf{linear}. 
This additional property allows us to obtain the following 
\textbf{main results}: 
\begin{itemize}
\item[\bf R1] We reveal the \BN\ iteration as a \emph{weighted mean ergodic iteration} 
(see \cref{p:3.1});
\item[\bf R2] We identify the weak limit of the \BN\ iteration as the \emph{projection
of the initial point} onto the fixed point set of $T$ (see \cref{p:3.2});
\item[\bf R3] We show that the convergence is \emph{strong} when $\alpha=4$ 
(see \cref{t:super4}). 
\end{itemize}

The rest of the paper is organized as follows. 
In \cref{s:ergodic}, we review key weighted mean ergodic theorems 
and corresponding notions of regularity, the Cohen condition, 
and the Lorentz condition. The topic of \cref{s:betabin} is 
the discrete (symmetric) beta-binomial distribution whose 
corresponding matrix is shown to be regular and to satisfy the Lorentz condition.
In \cref{s:main}, we present the proofs of our main results 
\textbf{R1} and \textbf{R2}. 
The special case when $\alpha=4$ is analyzed in full detail in \cref{s:alpha4} 
where we also establish \textbf{R3}, the strong convergence of the \BN\ iteration. 
In the final \cref{s:beyond4}, we provide some observations, conjectures, 
and a direction for future research.

The notation used in this paper is fairly standard and 
largely follows \cite{BC2017}.

\section{Weighted Mean Ergodic Theorem and Regular Matrices}

\label{s:ergodic}

\subsection*{The Weak Limit of a Weighted Mean Ergodic Iteration}

We start with the following result concerning the identification of a
weak limit.

\begin{lemma}[abstract weak limit identification]
\label{dalimit}
Let $(T_n)_\nnn$ be a sequence of nonexpansive operators on $X$, 
let $x_0\in X$ and generate $(x_n)_\nnn$ via 
\begin{equation}
(\forall\nnn)\quad
x_{n+1} := T_nx_0,
\end{equation}
and let $F$ be a closed linear subspace of $X$. 
Assume that 
\begin{equation}
x_n\weakly Qx_0 \in F \subseteq \bigcap_\nnn \Fix T_n.
\end{equation}
Then 
$(\forall f\in F)(\forall\nnn)$ $\|x_{n}-f\|\leq \|x_0-f\|$, and
\begin{equation}
\label{e:quasi2}
    Qx_0 = P_Fx_0.
\end{equation}
\end{lemma}
\begin{proof}
Let $f\in F$. Then 
$(\forall\nnn)$ 
$\|x_{n+1}-f\| = \|T_nx_0-T_nf\|\leq \|x_0-f\|$ 
by the nonexpansiveness of $T_n$ and the assumption 
that $f\in \Fix T_n$.
Taking the weak limit, we obtain 
\begin{equation}
\label{e:quasi1}
(\forall f\in F)\quad 
\|Qx_0-f\|\leq \varliminf_{n\to\infty} \|x_{n+1}-f\| \leq \|x_0-f\|. 
\end{equation}
For every $t\in\RR$, set $f_t := (1-t)P_Fx_0 + tQx_0 \in F$. 
Now substitute this $f_t$ into \cref{e:quasi1}, then square and expand to obtain
$(1-t)^2\|Qx_0-P_Fx_0\|^2 \leq \|P_{F^\perp}x_0\|^2+t^2\|P_Fx_0-Qx_0\|^2$;
thus, 
\begin{equation*}
(1-2t)\|Qx_0-P_Fx_0\|^2 \leq \|P_{F^\perp}x_0\|^2. 
\end{equation*}
Letting $t\to-\infty$ yields $Qx_0 = P_Fx_0$, i.e., \cref{e:quasi2}.
\end{proof}

\begin{remark}
Let $C$ be a nonempty closed convex subset of $X$, and let $x_0\in X$. The set 
\begin{equation}
\mathcal{Q}_C(x_0) 
:= \menge{y\in C}{(\forall c\in C)\;\;\|c-y\|\leq \|c-x_0\|}
\end{equation}
was termed the 
\emph{quasi-projection of $x_0$ on $C$} 
by Baillon and Bruck in \cite[Page~19]{BaiBru}; 
they pointed out that it ``is the set of points 
which are candidate images of $x_0$ in retracting 
$C\cup\{x_0\}$ onto $C$ quasi-nonexpansively''. 
Using this notion, we see that 
\cref{e:quasi1} states that 
$Qx_0\in \mathcal{Q}_F(x_0)$. 
Because $F$ is a closed linear subspace, it follows from 
\cite[Proposition~3.4.4]{Thesis} that 
$\mathcal{Q}_F(x_0)=\{P_{F}(x_0)\}$ --- for 
the reader's convenience, we have included this
argument in the proof of \cref{dalimit}. 
\end{remark}

\begin{corollary}[weak limit of a mean ergodic iteration]
\label{Cohenlim}
Let $T\colon X\to X$ be linear and nonexpansive, let 
$C = (c_{n,k})_{(n,k)\in\NN\times\NN}\in\RR^{\NN\times\NN}$, 
and assume that 
\begin{enumerate}
\item 
$(\forall\nnn)$ $\displaystyle\sum_{k=0}^\infty c_{n,k}=1$,\\[+1mm]
\item 
$(\forall\nnn)(\forall\kkk)$ $c_{n,k}\geq 0$,\\[+1mm]
\item 
$(\forall x_0\in X)$ \;\;
$\displaystyle \mathrm{weak}\lim_{n\to\infty} \sum_{k=0}^{\infty} c_{n,k}T^kx_0 \in \Fix T$.
\end{enumerate}
Then for every $x_0\in X$, we have
\begin{equation}
\sum_{k=0}^\infty c_{n,k}T^kx_0 \weakly P_{\Fix T}(x_0).
\end{equation}
\end{corollary}
\begin{proof}
This follows from \cref{dalimit} 
by setting $T_n := \sum_{k=0}^\infty c_{n,k}T^k$ and noting that $T_n$ is nonexpansive with 
$F = \Fix T \subseteq \Fix T_n$ 
\end{proof}

\subsection*{Regular Matrices and a Convergent 
Weighted Mean Ergodic Iteration}

We now turn to the notion of a regular matrix, and the conditions 
by Cohen and by Lorentz. These are instrumental in 
establishing convergence for weighted mean ergodic iterations. 

\begin{definition}[regular matrix]
Let $C = (c_{n,k})_{(n,k)\in\NN\times\NN}\in\RR^{\NN\times\NN}$. 
Then $C$ is 
        \emph{regular} if for every convergent sequence of real numbers $(\xi_n)_\nnn$, one has 
        \begin{align*}
            \lim_{n}\sum_{k=0}^\infty c_{n,k}\xi_k=\lim_{n}\xi_n. 
        \end{align*}
    \end{definition}

\begin{fact}[Silverman-Toeplitz]   
\label{f:ST}
Let $C = (c_{n,k})_{(n,k)\in\NN\times\NN}\in\RR^{\NN\times\NN}$. 
Then $C$ is regular if and only if the following three conditions 
hold simultaneously:
\begin{enumerate}[leftmargin=3.8em,itemsep=0.5em]
\item[\bf (ST1)] $\displaystyle \sup_\nnn \sum_{k=0}^\infty |c_{n,k}| <\pinf$. 
\item[\bf (ST2)] $\displaystyle (\forall\kkk)$ 
$\displaystyle\lim_{n\to\infty} c_{n,k}=0$. 
\item[\bf (ST3)] $\displaystyle \lim_{n\to\infty}\sum_{k=0}^\infty c_{n,k}=1$.
\end{enumerate}
\end{fact}
\begin{proof}
See the historical references \cite{Toeplitz} or \cite{Silverman}, 
and, e.g., \cite[Theorem~1.1]{ADN} for a fairly recent presentation.
\end{proof}

\begin{definition}[Cohen condition and Lorentz condition]
Let $C = (c_{n,k})_{(n,k)\in\NN\times\NN}\in\RR^{\NN\times\NN}$. 
Then the Cohen condition {\rm \textbf{(C)}}  
and the Lorentz condition {\rm \textbf{(L)}} 
are the following: 
\begin{enumerate}[leftmargin=3.8em,itemsep=0.5em]
\item[\bf (C)] $\displaystyle \lim_{K\to\infty}\sup_{\nnn}
\sum_{k=K}^\infty |c_{n,k}-c_{n,k+1}| = 0$. 
\item[\bf (L)] $\displaystyle \lim_{n\to\infty}
\sum_{k=0}^\infty |c_{n,k}-c_{n,k+1}| = 0$. 
\end{enumerate}
\end{definition}

The Cohen condition is important to us because of the following:

\begin{fact}[Cohen] 
{\rm \cite[Theorem]{Cohen}}
\label{f:Cohen}
Let $T\colon X\to X$ be linear and nonexpansive. 
Let $C = (c_{n,k})_{(n,k)\in\NN\times\NN}\in\RR^{\NN\times\NN}$ be regular
and assume the Cohen condition holds. 
Then for every $x_0$, we have the following strong convergence:
\begin{equation}
\label{e:Cohen}
\lim_{n\to\infty} \sum_{k=0}^{\infty} c_{n,k}T^kx_0 \in \Fix T.
\end{equation}
\end{fact}

The Cohen condition \textbf{(C)} is 
a little bit awkward to verify. 
It turns out that one can work instead with the 
more pleasant Lorentz condition \textbf{(L)}. 
To see this, we need the following: 

\begin{proposition}
\label{p:helps}
Let $C = (c_{n,k})_{(n,k)\in\NN\times\NN}\in\RR^{\NN\times\NN}$. 
Then we have the following implications:
\begin{enumerate}[itemsep=0.5em]
\item 
\label{p:helps1}
Suppose that $C$ satisfies {\rm\textbf{(ST2)}} and {\rm\textbf{(C)}}. 
Then $C$ satisfies {\rm\textbf{(L)}}.
\item 
\label{p:helps2}
Suppose that $C$ satisfies {\rm\textbf{(L)}} 
and $(\forall\nnn)$ $\displaystyle\sum_{k=0}^\infty |c_{n,k}|<\pinf$.
Then $C$ satisfies  {\rm\textbf{(C)}}.
\end{enumerate}
\end{proposition}
\begin{proof}
\cref{p:helps1}: 
Let $\varepsilon>0$. 
From \textbf{(C)}, we obtain $K\in\NN$ such that 
\begin{equation}
\label{e:260324a}
\sup_\nnn \sum_{k=K}^\infty |c_{n,k}-c_{n,k+1}| < \thalb\varepsilon. 
\end{equation}
If $k\in\{0,1,\ldots,K-1\}$, 
then \textbf{(ST2)} implies 
$\lim_{n\to\infty}c_{n,k}=0$ and
$\lim_{n\to\infty}c_{n,k+1}=0$; 
hence, $\lim_{n\to\infty}|c_{n,k}-c_{n,k+1}| = 0$.
It follows that 
$\lim_{n\to\infty}
\sum_{k=0}^{K-1} |c_{n,k}-c_{n,k+1}| = 0$. 
Hence there exists $N\in\NN$ such that 
\begin{equation}
\label{e:260324b}
(\forall n\geq N)\quad 
\sum_{k=0}^{K-1} |c_{n,k}-c_{n,k+1}| < \thalb\varepsilon. 
\end{equation}
Combining \cref{e:260324a} with \cref{e:260324b}, 
we deduce that 
\begin{equation*}
(\forall n\geq N)\quad 
\sum_{k=0}^{\infty} |c_{n,k}-c_{n,k+1}| < \varepsilon,
\end{equation*}
which yields \textbf{(L)}. 

\cref{p:helps2}: 
If $\nnn$, then, 
by the summability assumption, 
$(c_{n,k})_{\kkk}$ and 
$(c_{n,k+1})_{\kkk}$ both belong to $\ell^1(\NN)$; hence, 
$(c_{n,k}-c_{n,k+1})_{\kkk} \in \ell^1(\NN)$ as well. 
In turn, this implies 
\begin{equation}
\label{e:260324c}
(\forall\nnn)\quad 
\lim_{K\to\infty}
\sum_{k=K}^\infty |c_{n,k}-c_{n,k+1}| = 0. 
\end{equation}
Now let $\varepsilon > 0$. 
By \textbf{(L)},  
\begin{equation*}
(\exists N\geq 1)(\forall n\geq N)\quad
\sum_{k=0}^\infty |c_{n,k}-c_{n,k+1}| < \varepsilon. 
\end{equation*}
Consequently, 
\begin{equation}
\label{e:260324f}
(\forall K\in\NN)(\forall n\geq N)\quad
\sum_{k=K}^\infty |c_{n,k}-c_{n,k+1}| < \varepsilon. 
\end{equation}
Now \cref{e:260324c} implies that 
\begin{equation}
\label{e:260324d}
(\forall n\in\{0,1,\ldots,N-1\})(\exists K_n\in\NN)\quad 
\sum_{k=K_n}^\infty |c_{n,k}-c_{n,k+1}| < \varepsilon. 
\end{equation}
Set 
\begin{equation*}
K := \max\{K_0,K_1,\ldots,K_{N-1}\}. 
\end{equation*}
Then \cref{e:260324d} yields 
\begin{equation}
\label{e:260324e}
(\forall n\in\{0,1,\ldots,N-1\})
\quad 
\sum_{k=K}^\infty |c_{n,k}-c_{n,k+1}| < \varepsilon.
\end{equation}
Combining \cref{e:260324f} and \cref{e:260324e}, 
we deduce that 
\begin{equation*}
(\forall\nnn)\quad 
\sum_{k=K}^\infty |c_{n,k}-c_{n,k+1}| < \varepsilon. 
\end{equation*}
Therefore, 
\begin{equation*}
\sup_\nnn\sum_{k=K}^\infty |c_{n,k}-c_{n,k+1}| \leq \varepsilon, 
\end{equation*}
and this gives us 
\textbf{(C)}. 
\end{proof}

\begin{corollary}
\label{c:CL}
Let $C = (c_{n,k})_{(n,k)\in\NN\times\NN}\in\RR^{\NN\times\NN}$ be regular.
Then: 
$C$ satisfies {\rm\textbf{(C)}}
$\Leftrightarrow$ 
$C$ satisfies~{\rm\textbf{(L)}}. 
\end{corollary}
\begin{proof}
``$\Rightarrow$'': This follows from \cref{p:helps}\cref{p:helps1}. 
``$\Leftarrow$'': Note that {\rm\textbf{(ST1)}} clearly implies 
that $(\forall\nnn)$ $\sum_{k=0}^\infty |c_{n,k}| < +\infty$. 
Hence this implication follows from 
\cref{p:helps}\cref{p:helps2}.
\end{proof}

We are now ready to record a result that will be convenient for future use:

\begin{theorem}
\label{cool}
Let $C = (c_{n,k})_{(n,k)\in\NN\times\NN}\in\RR^{\NN\times\NN}$ satisfy
the following:
\begin{subequations}
\begin{align}
&(\forall\nnn)(\forall\kkk)\quad c_{n,k}\geq 0,\label{cool1}\\
&(\forall\nnn)\quad \sum_{k=0}^\infty c_{n,k}=1,\label{cool2}\\
&(\forall\kkk)\quad\lim_{n\to\infty}c_{n,k}=0,\label{cool3}\\
&\lim_{n\to\infty}\sum_{k=0}^\infty |c_{n,k}-c_{n,k+1}| = 0.\label{cool4}
\end{align}
\end{subequations}
Let $T\colon X\to X$ be linear and nonexpansive. 
Then for every $x_0\in X$, we have 
\begin{equation}
\lim_{n\to\infty} \sum_{k=0}^{\infty} c_{n,k}T^kx_0 = P_{\Fix T}(x_0).
\end{equation}
\end{theorem}
\begin{proof}
Note that 
\textbf{(ST1)} holds because of \cref{cool1} and \cref{cool2}.
Next, \cref{cool3} is exactly \textbf{(ST2)}.
Moreover, \cref{cool2} clearly yields \textbf{(ST3)}. 
By \cref{f:ST}, the matrix $C$ is regular.

Now \cref{cool4} is precisely the Lorentz condition \textbf{(L)}. 
It thus follows from \cref{c:CL} that the Cohen condition \textbf{(C)} holds as well.

Let $x_0\in X$. \cref{f:Cohen} implies that 
\begin{equation*}
\lim_{n\to\infty} \sum_{k=0}^{\infty} c_{n,k}T^kx_0 \in \Fix T. 
\end{equation*}
The conclusion now follows from \cref{Cohenlim}. 
\end{proof}

\begin{remark}[connections to works by Reich, and by Browder and Brezis] \ 
\begin{enumerate}
\item
Reich obtained (see \cite[Theorem~1.(b)]{Reich78})  
a strong convergence result that is more general than \cref{cool} in that it allows for 
\emph{nonlinear} nonexpansive operators.
His proof is somewhat more complicated, and the strong limit
is identified as the \emph{asymptotic center} of $(T^nx_0)_\nnn$. 
In contrast, our derivation of \cref{cool} is a consequence of 
the classical Cohen \cref{f:Cohen} combined with \cref{dalimit}, 
and we also provide a more explicit formula for the limit, namely $P_{\Fix T}(x_0)$.

\item
Brezis and Browder (see \cite[Theorem~1]{BB76} showed that 
if we replace in \cref{cool} the assumption 
\cref{cool4} by the less restrictive condition
\begin{equation}
\lim_{n\to\infty} \sum_{k=0}^\infty \max\{0,c_{n,k+1}-c_{n,k}\} = 0,
\end{equation}
then $(\sum_{k=0}^{\infty} c_{n,k}T^kx_0)_\nnn$ converges weakly to a fixed point of $T$. 
In fact, their proof does not require the linearity of $T$.
See also their follow up paper \cite{BB77}. 
\end{enumerate}
\end{remark}

\begin{corollary}
\label{kool}
Let $C = (c_{n,k})_{(n,k)\in\NN\times\NN}\in\RR^{\NN\times\NN}$ satisfy
the following:
\begin{subequations}
\begin{align}
&(\forall\nnn)(\forall\kkk)\quad c_{n,k}\geq 0,\label{kool1}\\
&(\forall\nnn)\quad \sum_{k=0}^\infty c_{n,k}=1,\label{kool2}\\
&(\forall\kkk)\quad\lim_{n\to\infty}c_{n,k}=0,\label{kool3}\\
&\lim_{n\to\infty}\sum_{k=0}^\infty 
\max\{0,c_{n,k+1}-c_{n,k}\} = 0.\label{kool4}
\end{align}
\end{subequations}
Let $T\colon X\to X$ be linear and nonexpansive. 
Then for every $x_0\in X$, we have 
\begin{equation}
\mathrm{weak}\lim_{n\to\infty} \sum_{k=0}^{\infty} c_{n,k}T^kx_0 = P_{\Fix T}(x_0).
\end{equation}
\end{corollary}
\begin{proof}
As we just pointed out, the weak convergence follows from \cite[Theorem~1]{BB76}. 
In turn, the limit formula is again a consequence of \cref{Cohenlim}. 
\end{proof}

\section{The Beta-Binomial Distribution}

\label{s:betabin}

In this section, we take a scenic detour to discuss the discrete beta-binomial 
distribution. 

We start by reviewing the beta and gamma functions. 

\begin{definition}[beta function and gamma function]
For $x>0$ and $y>0$, the \emph{beta function} is defined by 
\begin{equation}
\label{e:defB}
\Bf(x,y) := \int_0^1 t^{x-1}(1-t)^{y-1}\, dt
\end{equation}
while the \emph{gamma function} is defined by 
\begin{equation}
\Gamma(x) := \int_0^\infty t^{x-1}e^{-t}\,dt. 
\end{equation}
\end{definition}

The gamma function can also 
be expressed as 
\begin{equation}
\Gamma(x) = \lim_{k\to\infty}
\frac{k!k^{x-1}}{x(x+1)\cdots(x+k-1)}; 
\end{equation}
see, e.g., \cite[Definition~1.1.1 and Corollary~1.1.5]{AAR}.

The following properties of the beta and gamma functions are known. 
\begin{fact}
\label{f:BG}
Let $x>0$, $y>0$, and $m,n\in\{1,2,\ldots\}$ with $m\leq n$. Then:
\begin{enumerate}
\item
\label{f:BG1}
$\Bf(x,y) = \Bf(y,x)$. 
\item
\label{f:BG2}
$\Bf(x,y) = \Gamma(x)\Gamma(y)/\Gamma(x+y)$.
\item
\label{f:BG3}
$\Gamma(x+1)=x\Gamma(x)$.
\item
\label{f:BGnew}
$\prod_{k=m}^n (k+x) = {\Gamma(n+1+x)}/{\Gamma(m+x)}$.
\item 
\label{f:BG4}
$\Gamma(1)=1$ and $\Gamma(n+1)=n!$
\item 
\label{f:BG5}
$\Gamma(n+\tfrac{1}{2}) = \sqrt{\pi}\cdot (2n)!/(4^n n!)$.
\end{enumerate}
\end{fact}
\begin{proof}
\cref{f:BG1}: Change variables via $s=1-t$. 
\cref{f:BG2}: \cite[Theorem~1.1.4]{AAR}. 
\cref{f:BG3}: \cite[Equation~(1.1.6)]{AAR}. 
\cref{f:BGnew}: 
Using \cref{f:BG3}, we have 
$\Gamma(n+1+x) = \Gamma(n+x)(n+x) = \cdots = \Gamma(m+x)\prod_{k=m}^n (k+x)$. 
\cref{f:BG4}: \cite[Equations~(1.1.7) and (1.1.8)]{AAR}. 
\cref{f:BG5}: This follows from the duplication formula 
for the Gamma function (see \cite[Theorem~1.5.1]{AAR}) combined with 
with the fact that $\Gamma(\tfrac{1}{2}) = \sqrt{\pi}$ (see \cite[(1.1.22)]{AAR}) 
\end{proof}

\begin{proposition}
Let $x>0$, $y>0$, and $z>1$. Then
\begin{subequations}
\begin{align}
\Bf(x+1,y) &= \frac{x}{x+y}\,\Bf(x,y), \label{e:B1}\\
\Bf(x,y+1) &= \frac{y}{x+y}\,\Bf(x,y), \label{e:B2}\\
\frac{\Bf(x+1,z-1)}{\Bf(x,z)} &= \frac{x}{z-1}. \label{e:B3}
\end{align}
\end{subequations}
\end{proposition}
\begin{proof}
\cref{e:B1}: 
Using \cref{f:BG}\cref{f:BG2}\&\cref{f:BG3}, 
we have
\begin{align*}
\Bf(x+1,y)
&= \frac{\Gamma(x+1)\Gamma(y)}{\Gamma(x+1+y)}
= \frac{x\Gamma(x)\Gamma(y)}{(x+y)\Gamma(x+y)}
= \frac{x}{x+y}\Bf(x,y),
\end{align*}
as claimed. 

\cref{e:B2}: 
Using \cref{f:BG}\cref{f:BG1} and \cref{e:B1},
we obtain 
\begin{align*}
\Bf(x,y+1) = \Bf(y+1,x) = \frac{y}{y+x}\Bf(y,x)=\frac{y}{x+y}\Bf(x,y),
\end{align*}
as needed.

\cref{e:B3}: 
Using \cref{e:B1} and \cref{e:B2}, 
we get 
\begin{align*}
\frac{\Bf(x+1,z-1)}{\Bf(x,z)}
&= \frac{x\Bf(x,z-1)}{(x+z-1)\Bf(x,(z-1)+1)}
= \frac{x\Bf(x,z-1)(x+z-1)}{(x+z-1)(z-1)\Bf(x,z-1)}
= \frac{x}{z-1},
\end{align*}
and we're done.
\end{proof}

For the remainder of this section, we let 
\begin{subequations}
\label{e:bnk}
\begin{empheq}[box=\mybluebox]{equation}
\beta > 1.
\end{empheq}
For every $n\in\ZZ$ and $k\in\ZZ$, 
\begin{empheq}[box=\mybluebox]{equation}
\label{e:bnk'}
 b_{n,k}:=\binom{n}{k}\frac{\Bf(k+\beta,n-k+\beta)}{\Bf(\beta,\beta)} >0,
 \quad
 \text{when $n\in\NN$ and $k\in\{0,1,\ldots,n\}$},
\end{empheq}
and 
\begin{empheq}[box=\mybluebox]{equation}
b_{n,k} := 0, \quad
\text{
when $n\in\ZZ\smallsetminus\NN$ or $k<0$ or $k>n$. 
}
\end{empheq}
\end{subequations}

The next result implies 
that the matrix $(b_{n,k})_{(n,k)\in\NN\times\NN}$ 
satisfies \textbf{(ST3)} and \textbf{(ST1)}.

\begin{lemma}
\label{l:b1}
For every $\nnn$, we have $\displaystyle\sum_{k\in\NN}b_{n,k}=1$. 
\end{lemma}
\begin{proof}
Let $\nnn$. Then
\begin{align}
\sum_{k\in\NN}b_{n,k}
&= 
\sum_{k=0}^n b_{n,k}
= \sum_{k=0}^n \binom{n}{k}\frac{\Bf(k+\beta,n-k+\beta)}{\Bf(\beta,\beta)}
\tag{by \cref{e:bnk}}\\
&= \frac{1}{\Bf(\beta,\beta)}
\sum_{k=0}^n \binom{n}{k} \int_0^1 t^{k+\beta-1}(1-t)^{n-k+\beta-1}\, dt
\tag{by \cref{e:defB}}\\
&= 
\frac{1}{\Bf(\beta,\beta)}
\int_0^1 t^{\beta-1}(1-t)^{\beta-1}
\sum_{k=0}^n \binom{n}{k} t^k(1-t)^{n-k}\,dt \notag\\
&= \frac{1}{\Bf(\beta,\beta)}
   \int_0^1 t^{\beta-1}(1-t)^{\beta-1}\,dt 
   \tag{by the binomial theorem}\\
&= 1, \tag{by \cref{e:defB}}
\end{align}
as claimed.
\end{proof}

Now we derive a row bound for 
the matrix $(b_{n,k})_{(n,k)\in\NN\times\NN}$ 
that will imply \textbf{(ST2)}.

\begin{lemma}[row bound]
\label{l:b2}
We have 
\begin{equation}
(\forall\nnn)\quad 
\max_{k\in\ZZ} b_{n,k} \leq \frac{4^{1-\beta}}{\Bf(\beta,\beta)}
\cdot \frac{1}{n+1};
\end{equation}
consequently,
\begin{equation}
(\forall k\in\ZZ)\quad \lim_{n\to\infty} b_{n,k} = 0. 
\end{equation}
\end{lemma}
\begin{proof}
Define 
\begin{equation}
\phi\colon [0,1]\to\RR\colon t\mapsto t^{\beta-1}(1-t)^{\beta-1}.
\label{e:phi}
\end{equation}
It's clear that $\phi$ is continuous. 
Because 
\begin{equation*}
\phi'(t) = (\beta-1)t^{\beta-2}(1-t)^{\beta-2}(1-2t)
\end{equation*}
on $\left]0,1\right[$, we see that
$\phi$ is \emph{increasing} on $[0,1/2]$ and
\emph{decreasing} on $[1/2,1]$.
Hence
\begin{equation}
\label{e:maxphi}
\max_{t\in[0,1]}\phi(t) = \phi(1/2) = (1/2)^{2(\beta-1)} = 4^{1-\beta}.
\end{equation}
Now let $\nnn$ and $k\in\{0,1,\ldots,n\}$. 
Then 
\begin{align}
b_{n,k}
&=
\binom{n}{k}\frac{\Bf(k+\beta,n-k+\beta)}{\Bf(\beta,\beta)}
\tag{by \cref{e:bnk}}\\
&=
\frac{1}{\Bf(\beta,\beta)}
\binom{n}{k} \int_0^1 t^{k+\beta-1}(1-t)^{n-k+\beta-1}\, dt
\tag{by \cref{e:defB}}\\
&=
\frac{1}{\Bf(\beta,\beta)}
\binom{n}{k} \int_0^1 t^{k}(1-t)^{n-k}\phi(t)\, dt
\tag{by \cref{e:phi}}\\
&\leq 
\frac{4^{1-\beta}}{\Bf(\beta,\beta)}
\binom{n}{k} \int_0^1 t^{k}(1-t)^{n-k}\, dt
\tag{by \cref{e:maxphi}}\\
&= 
\frac{4^{1-\beta}}{\Bf(\beta,\beta)}
\binom{n}{k} \Bf(k+1,n-k+1)
\tag{by \cref{e:defB}}\\
&= 
\frac{4^{1-\beta}}{\Bf(\beta,\beta)}
\binom{n}{k} 
\frac{\Gamma(k+1)\Gamma(n-k+1)}{\Gamma(n+2)}
\tag{by \cref{f:BG}\cref{f:BG2}}\\
&= 
\frac{4^{1-\beta}}{\Bf(\beta,\beta)}
\cdot 
\frac{n!}{k!(n-k)!}\cdot 
\frac{k!(n-k)!}{(n+1)!}
\tag{by \cref{f:BG}\cref{f:BG4}}\\
&= 
\frac{4^{1-\beta}}{\Bf(\beta,\beta)}
\cdot\frac{1}{n+1},
\notag
\end{align}
as announced.
The ``Consequently'' part is now clear.
\end{proof}

We have shown that 
the matrix $(b_{n,k})_{(n,k)\in\NN\times\NN}$ is regular. 
This matrix also enjoys the pleasing 
symmetry and unimodality properties: 

\begin{lemma}[symmetry and unimodality]
\label{l:b3}
For every $\nnn$, we have the symmetry
\begin{equation}
\label{e:bsymm}
(\forall k\in\ZZ)\quad b_{n,k}=b_{n,n-k}.
\end{equation}
Moreover, setting $m_n := \lfloor n/2\rfloor$, we have
the following unimodality:

\begin{subequations}
\begin{align}
0<b_{n,0}<b_{n,1}<\cdots <b_{n,m_n-1}<b_{n,m_n}
&>b_{n,m_n+1}>\cdots>b_{n,n}>0,
\quad
\text{if $n$ is even;}\label{e:ueven}\\
0<b_{n,0}<b_{n,1}<\cdots<b_{n,m_n-1}<b_{n,m_n}&=b_{n,m_n+1}>\cdots>b_{n,n}>0,
\quad
\text{if $n$ is odd.}\label{e:uodd}
\end{align}
\end{subequations}

\end{lemma}
\begin{proof}
Let $\nnn$.
When $k<0$ or $k>n$, then \cref{e:bsymm} 
turns into $0=0$; when 
$k\in\{0,\ldots,n\}$, then 
the symmetry follows from the symmetry of the binomial coefficients 
and of $\Bf$ (see \cref{f:BG}\cref{f:BG1}). 
We've verified \cref{e:bsymm}.

Next, fix temporarily $k\in\{0,1,\ldots,n-1\}$. 
Then 
\begin{align*}
\frac{b_{n,k+1}}{b_{n,k}}
&= 
\frac{\binom{n}{k+1}}{\binom{n}{k}}
\cdot
\frac{\Bf(k+1+\beta,n-k-1+\beta)}{\Bf(k+\beta,n-k+\beta)} 
\tag{by \cref{e:bnk}}\\
&= \frac{n!}{(k+1)!(n-k-1)!}
\cdot
\frac{k!(n-k)!}{n!}
\cdot
\frac{k+\beta}{n-k-1+\beta}
\tag{by \cref{e:B3}}\\
&= 
\frac{n-k}{k+1}\cdot
\frac{k+\beta}{n-k-1+\beta}.
\end{align*}
It follows that 
\begin{subequations}
\begin{align}
b_{n,k}<b_{n,k+1}
&\Leftrightarrow 
1 < \frac{b_{n,k+1}}{b_{n,k}}
\Leftrightarrow 
(k+1)(n-k-1+\beta)<(n-k)(k+\beta)
\end{align}
and 
\begin{align}
b_{n,k}=b_{n,k+1}
&\Leftrightarrow 
(k+1)(n-k-1+\beta)=(n-k)(k+\beta).
\end{align}
\end{subequations}
On the other hand,
$(n-k)(k+\beta)-(k+1)(n-k-1+\beta)= (\beta-1)(n-1-2k)$.
Altogether, because $\beta>1$, we deduce 
\begin{subequations}
\label{e:unim}
\begin{equation}
\label{e:unim<}
b_{n,k}<b_{n,k+1}
\Leftrightarrow
n-1-2k>0
\Leftrightarrow
k<\frac{n-1}{2}
\end{equation}
and 
\begin{equation}
\label{e:unim=}
b_{n,k}=b_{n,k+1}
\Leftrightarrow
k=\frac{n-1}{2}.
\end{equation}
\end{subequations}

\emph{Case~1:} $n$ is even.\\
Then $m_n = n/2$, $(n-1)/2\notin\ZZ$, 
and $\lfloor (n-1)/2\rfloor = (n-2)/2 = m_n-1$. 
Hence \cref{e:unim<} yields

\begin{equation}
0 < b_{n,0}<b_{n,1}<\cdots < b_{n,m_n-1}<b_{n,m_n}>b_{n,m_n+1};
\end{equation}

combining with the symmetry \cref{e:bsymm},
we obtain \cref{e:ueven}.

\emph{Case~2:} $n$ is odd.\\
Then $m_n = (n-1)/2\in\NN$, 
and \cref{e:unim} yields

\begin{equation}
0 < b_{n,0}<b_{n,1}<\cdots< b_{n,m_n-1}<b_{n,m_n}=b_{n,m_n+1};
\end{equation}

combining with the symmetry \cref{e:bsymm},
we obtain \cref{e:uodd}.
\end{proof}

\begin{lemma}[Lorentz condition]
\label{l:b4}
For every $\nnn$, set $m_n := \lfloor n/2\rfloor$. 
Then 
\begin{equation}
\sum_{k\in\NN} |b_{n,k}-b_{n,k+1}| = 2b_{n,m_n}-b_{n,0} < 
\frac{4^{1-\beta}}{\Bf(\beta,\beta)}\cdot \frac{2}{n+1} 
\to 0 \quad\text{as $n\to\infty$.}
\end{equation}
\end{lemma}
\begin{proof}
First, the announced inequality is clear from \cref{l:b2}.
Secondly, we set 
\begin{align*}
\sigma_n 
&:= \sum_{k\in\NN} |b_{n,k+1}-b_{n,k}|
=\Big(\sum_{k=0}^{n-1}|b_{n,k}-b_{n,k+1}|\Big)
+ |b_{n,n}-0| 
= \Big(\sum_{k=0}^{n-1}|b_{n,k}-b_{n,k+1}|\Big) + b_{n,n}.
\end{align*}

\emph{Case~1:} $n$ is even.\\
Then \cref{e:ueven} yields 
\begin{align*}
\sigma_n 
&= \Big(\sum_{k=0}^{m_n-1}(b_{n,k+1}-b_{n,k})\Big)
+ \Big(\sum_{k=m_n}^{n-1}(b_{n,k}-b_{n,k+1})\Big) + b_{n,n}
= \big(b_{n,m_n}-b_{n,0}\big)
+\big(b_{n,m_n}-b_{n,n}\big) + b_{n,n}\\
&= 2b_{n,m_n}-b_{n,0},
\end{align*}
as claimed.

\emph{Case~2:} $n$ is odd.\\
Then \cref{e:uodd} implies $b_{n,m_n}=b_{n,m_n+1}$ and 
\begin{align*}
\sigma_n 
&= \Big(\sum_{k=0}^{m_n-1}(b_{n,k+1}-b_{n,k})\Big)
+ \Big(\sum_{k=m_n+1}^{n-1}(b_{n,k}-b_{n,k+1})\Big) + b_{n,n}
= \big(b_{n,m_n}-b_{n,0}\big)
+\big(b_{n,m_n+1}-b_{n,n}\big) + b_{n,n}\\
&= 2b_{n,m_n}-b_{n,0},
\end{align*}
and we're done.
\end{proof}

Let us summarize the results obtained so far 
in the following: 

\begin{theorem}
Recall our assumptions in \cref{e:bnk} in this section. 
Let $T\colon X\to X$ be linear and nonexpansive, and 
let $x_0\in X$. 
Then 
\begin{equation}
\lim_{n\to\infty} \sum_{k=0}^{\infty} b_{n,k}T^kx_0 = P_{\Fix T}(x_0).
\end{equation}
\end{theorem}
\begin{proof}
\cref{l:b1}, \cref{l:b2}, and \cref{l:b4} show that the 
matrix $(b_{n,k})_{(n,k)\in\NN\times\NN}$ satisfies the 
assumptions of \cref{cool}, and hence we're done.
\end{proof}

We conclude this section by introducing the following sequence of 
polynomials induced by the beta-binomial distribution:

\begin{empheq}[box=\mybluebox]{equation}
\label{e:qn}
(\forall\nnn)\quad q_n(t) := \sum_{k=0}^n b_{n,k}t^k.
\end{empheq}
Note that 
\begin{equation}
\label{e:q01}
q_0(t) = 1\quad\text{and}\quad
q_1(t) = \thalb + \thalb t. 
\end{equation}

The proof of the following recursion was obtained with the help of 
\texttt{ChatGPT} which allowed us to find an elementary proof that 
avoided the use of Gaussian hypergeometric functions for the 
beta-binomial distribution (see \cite[Section~6.2.2]{JKK}). 

\begin{proposition}[recursion formula for $q_n$] 
\label{p:qrecursion}
We have, for all $n\geq 1$, the recursion 
\begin{equation}
\label{e:qrecursion}
(n+2\beta)q_{n+1}(t) = (n+\beta)(1+t)q_n(t) - ntq_{n-1}(t).
\end{equation}
\end{proposition}
\begin{proof}
The formula is clear for $t=1$, so we assume 
that $t\neq 1$. 
Let $m\in\NN$ and set 
\begin{equation}
s(x) := 1-x+tx. 
\end{equation}
Then  \cref{e:bnk'}, \cref{e:defB}, and the binomial theorem yield
\begin{subequations}
\label{e:qm}
\begin{align}
q_m(t)
&=\sum_{k=0}^m \binom{m}{k}\frac{\Bf(k+\beta,m-k+\beta)}{\Bf(\beta,\beta)}t^k\\
&=\frac{1}{\Bf(\beta,\beta)}\sum_{k=0}^m \binom{m}{k}
\int_0^1 x^{k+\beta-1}(1-x)^{m-k+\beta-1}t^k\,dx\\
&=\frac{1}{\Bf(\beta,\beta)}\int_0^1 x^{\beta-1}(1-x)^{\beta-1}
\sum_{k=0}^m \binom{m}{k}(1-x)^{m-k}(tx)^k\,dx\\
&=\frac{1}{\Bf(\beta,\beta)}\int_0^1 x^{\beta-1}(1-x)^{\beta-1} (s(x))^m\,dx.
\end{align}
\end{subequations}
Now let $n\geq 1$ and consider the function 
\begin{equation*}
h(x) := x^{\beta}(1-x)^{\beta}(s(x))^n,
\end{equation*}
which is a product of the three functions 
$x^\beta, (1-x)^\beta$, and $(s(x))^n$.
Now $s'(x)=t-1$, so the product rule (for 3 functions) yields that 
\begin{align*}
h'(x)
&= x^{\beta-1}(1-x)^{\beta-1}(s(x))^{n-1}
\Big(\beta(1-x)s(x)+x\beta(-1)s(x)+x(1-x)n(t-1)\Big)\\
&= x^{\beta-1}(1-x)^{\beta-1}(s(x))^{n-1}
\Big(\beta(1-2x)s(x)+n(t-1)x(1-x)\Big). 
\end{align*}
On the other hand, we have 
the following alternative expression, 
proposed by \texttt{ChatGPT}, 
for the term in the big parentheses:
\begin{align*}
\beta(1-2x)s(x)+n(t-1)x(1-x) 
&= 
-\frac{n+2\beta}{t-1}(s(x))^2+\frac{(n+\beta)(1+t)}{t-1}s(x)-\frac{nt}{t-1}.
\end{align*}
Altogether, 
\begin{align*}
(t-1)h'(x)
&= x^{\beta-1}(1-x)^{\beta-1}s(x)^{n-1}
\Big(-(n+2\beta)(s(x))^2 
+(n+\beta)(1+t)s(x)-nt\Big). 
\end{align*}
Because $h(0)=h(1)=0$, 
the Fundamental Theorem of Calculus and \cref{e:qm} yield
\begin{align*}
0 
&= 
(t-1)\int_0^1 h'(x)\,dx\\
&= 
\int_0^1 x^{\beta-1}(1-x)^{\beta-1}s(x)^{n-1}
\Big(-(n+2\beta)(s(x))^2 
+(n+\beta)(1+t)s(x)-nt\Big)\,dx\\
&= \Bf(\beta,\beta)\Big(-(n+2\beta)q_{n+1}(t) + (n+\beta)(1+t)q_n(t)-ntq_{n-1}(t)\Big).
\end{align*}
Dividing by $\Bf(\beta,\beta)$ followed by re-arranging, 
we obtain \cref{e:qrecursion}. 
\end{proof}

\section{\BN\ Acceleration, Regularity of the Coefficients, 
and Limit Identification} 

\label{s:main}

Throughout this section, we let 
\begin{subequations}
\label{e:fastKM} 
\begin{empheq}[box=\mybluebox]{equation}
\alpha> 2,
\end{empheq}
and assume that 
\begin{empheq}[box=\mybluebox]{equation}
T\colon X\to X 
\;\;\text{is linear and nonexpansive, and $x_0\in X$ is a starting point} 
\end{empheq}
for the sequence $(x_n)_\nnn$ generated by 
the \emph{\BN\ acceleration} 
\begin{empheq}[box=\mybluebox]{align}
x_1 &:= x_0, \label{e:fastKM1}\\
(\forall n\geq 1)\quad 
x_{n+1} &:= \left(1-\frac{\alpha}{2(n+\alpha)}\right)x_n
+\frac{\alpha}{2(n+\alpha)}Tx_n+\frac{n}{n+\alpha}(Tx_n-Tx_{n-1}). 
\label{e:fastKM2}
\end{empheq}
The linearity of $T$ allows us to write 
\begin{empheq}[box=\mybluebox]{align}
(\forall\nnn)\quad  x_{n+1}=\sum_{k=0}^{\infty}c_{n,k}T^kx_0 =: p_n(T)x_0,
\label{e:fastKMc}
\end{empheq}
where we set 
$c_{n,k}:=0$ if $n<0$ or $k<0$ or $k>n$ 
and 
where $p_n(t)$ is a polynomial of degree $n$. 
\end{subequations}
The update formula \cref{e:fastKM2} yields the following recursion 
for the coefficients $c_{n,k}$:

\begin{lemma}[recursion formula for $c_{n,k}$]
\label{f:2.2}
Consider the assumptions \cref{e:fastKM}.
Then 
\begin{equation}
\label{e:cnk01}
c_{0,0} = 1, \quad
c_{1,0} = 1-\frac{\alpha}{2(1+\alpha)} = \frac{2+\alpha}{2(1+\alpha)}, \quad
c_{1,1} = \frac{\alpha}{2(1+\alpha)} 
\end{equation}
and 
\begin{subequations}
\label{e:cnk}
\begin{align}
(\forall n\geq 2)\quad 
c_{n,k}&=\left(1-\frac{\alpha}{2(n+\alpha)}\right)c_{n-1,k}
       +\frac{\alpha}{2(n+\alpha)}c_{n-1, k-1}
       +\frac{n}{n+\alpha}c_{n-1, k-1}-\frac{n}{n+\alpha}c_{n-2, k-1}
\label{e:cnklong}\\
&= \frac{2n+\alpha}{2(n+\alpha)}c_{n-1,k}
   + \frac{2n+\alpha}{2(n+\alpha)}c_{n-1,k-1}-\frac{n}{n+\alpha}c_{n-2,k-1}.
\label{e:cnkshort}
        \end{align}
\end{subequations}
\end{lemma}
\begin{proof}
By \cref{e:fastKM1}, $x_1=x_0=c_{0,0}T^0x_0 = c_{0,0}x_0$, so $c_{0,0}=1$. 
Next, by \cref{e:fastKM2}, 
\begin{subequations}
\begin{align*}
    x_2&=\left(1-\frac{\alpha}{2(1+\alpha)}\right)x_1
        +\frac{\alpha}{2(1+\alpha)}Tx_1+\frac{1}{1+\alpha}(Tx_1-Tx_0)
=\left(1-\frac{\alpha}{2(1+\alpha)}\right)x_0
+\frac{\alpha}{2(1+\alpha)}Tx_0.
\end{align*}
\end{subequations}
Therefore, 
\begin{align*}
    c_{1,0}&=1-\frac{\alpha}{2(1+\alpha)}=\frac{2+\alpha}{2(1+\alpha)}
    \quad\text{and}\quad 
    c_{1,1}=\frac{\alpha}{2(1+\alpha)},
\end{align*}
which yields \cref{e:cnk01}.

Before we prove \cref{e:cnk} by induction, let us note that 
\cref{e:cnkshort} is clear by algebraic simplification. 
For $n=2$, by \cref{e:fastKM2},
\begin{subequations}
\begin{align*}
x_3
&=
\left(1-\frac{\alpha}{2(2+\alpha)}\right)x_2
+\frac{\alpha}{2(2+\alpha)}Tx_2+\frac{2}{2+\alpha}(Tx_2-Tx_1)\\ 
&=
\frac{4+\alpha}{2(2+\alpha)}x_2 
+ \frac{4+\alpha}{2(2+\alpha)}Tx_2
-\frac{2}{2+\alpha}Tx_1\\
&=
\frac{4+\alpha}{2(2+\alpha)}(c_{1,0}x_0+c_{1,1}Tx_0)
+ \frac{4+\alpha}{2(2+\alpha)}(c_{1,0}Tx_0+c_{1,1}T^2x_0)
-\frac{4}{2(2+\alpha)}Tx_0\\
&=
\frac{4+\alpha}{2(2+\alpha)}c_{1,0}x_0
+ \frac{1}{2(2+\alpha)}\Big((4+\alpha)c_{1,1} + (4+\alpha)c_{1,0}-4\Big)Tx_0
+ \frac{4+\alpha}{2(2+\alpha)}c_{1,1}T^2x_0\\
&=
\frac{4+\alpha}{4(1+\alpha)}x_0
+ \frac{1}{4(2+\alpha)(1+\alpha)}
\big(2\alpha(1+\alpha)\big)Tx_0
+ \frac{\alpha(4+\alpha)}{4(2+\alpha)(1+\alpha)}T^2x_0\\
&=
\frac{4+\alpha}{4(1+\alpha)}x_0
+ \frac{\alpha}{2(2+\alpha)}
Tx_0
+ \frac{\alpha(4+\alpha)}{4(2+\alpha)(1+\alpha)}T^2x_0.
\end{align*}
\end{subequations}
Hence
\begin{subequations}
\label{e:c2k}
\begin{align}
c_{2,0} &= \frac{4+\alpha}{4(1+\alpha)}
= \frac{2\cdot 2+\alpha}{2(2+\alpha)}c_{1,0},\\
c_{2,1} &= \frac{\alpha}{2(2+\alpha)} 
= \frac{2\cdot 2+\alpha}{2(2+\alpha)}c_{1,1}
+ \frac{2\cdot 2+\alpha}{2(2+\alpha)}c_{1,0} - 
\frac{2}{2+\alpha}c_{0,0},\\
c_{2,2} &= \frac{\alpha(4+\alpha)}{4(2+\alpha)(1+\alpha)}
= \frac{2\cdot 2+\alpha}{2(2+\alpha)}c_{1,1},
\end{align}
\end{subequations}
which verifies the base case $n=2$. 

Now assume that \cref{e:cnk} holds for all $n\leq m$,  where $m\geq 2$. 
Then,  by \cref{e:fastKM},
\begin{align*}
x_{m+2}
&=
\left(1-\frac{\alpha}{2(m+1+\alpha)}\right)x_{m+1}
        +\frac{\alpha}{2(m+1+\alpha)}Tx_{m+1}+\frac{m+1}{m+1+\alpha}
        (Tx_{m+1}-Tx_{m})\\
&=
\left(1-\frac{\alpha}{2(m+1+\alpha)}\right)\sum_{k=0}^{m}c_{m,k}T^kx_0 
        +\left(\frac{\alpha}{2(m+1+\alpha)}+\frac{m+1}{m+1+\alpha}\right)
        \sum_{k=0}^{m}c_{m,k}T^{k+1}x_0\\ 
        &\qquad -\frac{m+1}{m+1+\alpha}\sum_{k=0}^{m-1}c_{m-1,k}T^{k+1}x_0\\
&=\sum_{k=0}^m\left(1-\frac{\alpha}{2(m+1+\alpha)}\right)c_{m,k}T^kx_0
        +\sum_{k=1}^{m+1}\left(\frac{\alpha}{2(m+1+\alpha)}+\frac{m+1}{m+1+\alpha}\right)c_{m,k-1}T^kx_0\\ 
        &\qquad -\sum_{k=1}^{m}\frac{m+1}{m+1+\alpha}c_{m-1,k-1}T^kx_0.     
\end{align*}
By the convention, we have
\begin{align*}
    c_{m+1, k}=\left(1-\frac{\alpha}{2(m+1+\alpha)}\right)c_{m,k}
            +\frac{\alpha}{2(m+1+\alpha)}c_{m, k-1}
            +\frac{m+1}{m+1+\alpha}c_{m, k-1}-\frac{m+1}{m+1+\alpha}c_{m-1, k-1},
\end{align*}
as claimed.
\end{proof}

\begin{remark}[recursion formula for $p_n$]
\label{r:pnrecursion}
In terms of the polynomials $p_n$ defined by \cref{e:fastKMc}, 
we note that \cref{f:2.2} turns into the following recursion formula:
\begin{subequations}
\begin{align}
p_0(t) &= 1,\\
p_1(t) &= \frac{2+\alpha}{2(1+\alpha)}+\frac{\alpha}{2(1+\alpha)}t,\\
(\forall n\geq 2)\quad p_{n}(t) &= 
\frac{2n+\alpha}{2(n+\alpha)}p_{n-1}(t)
+ \frac{2n+\alpha}{2(n+\alpha)}tp_{n-1}(t)-\frac{n}{n+\alpha}tp_{n-2}(t). 
\label{e:pnrecn}
\end{align}
\end{subequations}
\end{remark}

\begin{remark}[constant and leading coefficients of $p_n$]
\label{r:cn0}
It follows from \cref{f:2.2} that for $n\geq 2$, we have 
\begin{align}
\label{e:rcn0}
c_{n,0}=\frac{2n+\alpha}{2(n+\alpha)}c_{n-1,0};
\end{align}
luckily, after recalling \cref{e:cnk01}, we have $c_{1,0}= (2+\alpha)/(2(1+\alpha)) = (2\cdot 1+\alpha)/(2(1+\alpha))c_{0,0}$ as well. 
Thus 
\begin{align}
\label{e:260325a}
(\forall n\geq 1)\quad c_{n,0}
=\prod_{m=1}^n\frac{2m+\alpha}{2(m+\alpha)} 
=\frac{2+\alpha}{2(1+\alpha)}\prod_{m=2}^n\frac{2m+\alpha}{2(m+\alpha)}>0
\end{align}
It also follows from \cref{f:2.2} that for $n\geq 2$, we have 
        \begin{align*}
c_{n,n}=\frac{2n+\alpha}{2(n+\alpha)}c_{n-1,n-1};
        \end{align*}
because $c_{1,1} = \alpha/(2(1+\alpha))$, we obtain 
\begin{equation}
\label{e:cnn}
(\forall n\geq 2)\quad c_{n,n}=\frac{\alpha}{2(1+\alpha)}\prod_{m=2}^n\frac{2m+\alpha}{2(m+\alpha)}>0. 
\end{equation}
Combining with \cref{e:260325a}, we obtain
\begin{equation}
\label{e:cnn1}
(\forall n\geq 1)\quad c_{n,n}=\frac{\alpha}{2+\alpha}c_{n,0},
\end{equation}
where the case $n=1$ is verified directly. 
\end{remark}

\begin{theorem}[\BN\ distribution]
\label{p:3.1}
Let $(c_{n,k})_{n\in\NN,k\in\NN}$ be given by \cref{e:fastKMc} and  
\cref{f:2.2}. Then the following hold:
\begin{enumerate}[itemsep=0.5em]
\item 
\label{p:3.1i}
$(\forall n\in \NN) \quad \displaystyle \sum_{k=0}^{n}c_{n,k}=1$.
\item 
\label{p:3.1ii}
$(\forall n\in \NN)
(\forall k\in\{0,1,\ldots,n\}) \quad c_{n,k}> 0$.
\item 
\label{p:3.1iii}
$(\forall k\in \NN)\quad \displaystyle\lim_{n\to\infty} c_{n,k}=0$.
\end{enumerate}
Consequently, the coefficients $c_{n,k}$ correspond to 
a discrete probability distribution, which we call the \emph{\BN\ distribution}, and 
the iteration \cref{e:fastKM} is a special weighted mean iteration induced by the \BN\ distribution.
\end{theorem}
\begin{proof}
\cref{p:3.1i}:
For $n\in\{0,1\}$, 
it is obvious that the sum equals 1 by \cref{e:cnk01}.
Note \cref{e:cnk} yields for $n\geq 2$
\begin{subequations}
            \label{e:3.11}
\begin{align}
           \sum_{k=0}^nc_{n,k}&=\sum_{k=0}^n\left(\left(1-\frac{\alpha}{2(n+\alpha)}\right)c_{n-1,k}
            +\left(\frac{\alpha}{2(n+\alpha)}+\frac{n}{n+\alpha}\right)c_{n-1, k-1}
            -\frac{n}{n+\alpha}c_{n-2, k-1}\right)\\ 
            &=\sum_{k=0}^n\left(\frac{2n+\alpha}{2(n+\alpha)}c_{n-1,k}
            +\frac{2n+\alpha}{2(n+\alpha)}c_{n-1, k-1}-\frac{n}{n+\alpha}c_{n-2, k-1}\right)\\ 
            &=\frac{2n+\alpha}{2(n+\alpha)}\left(\sum_{k=0}^{n-1}c_{n-1,k}+\sum_{k=1}^nc_{n-1, k-1}\right)
            -\frac{n}{n+\alpha}\sum_{k=1}^{n-1}c_{n-2, k-1}\\ 
            &=\frac{2n+\alpha}{n+\alpha}\sum_{k=0}^{n-1}c_{n-1,k}-\frac{n}{n+\alpha}\sum_{k=0}^{n-2}c_{n-2, k}.
        \end{align}
\end{subequations}
Because $\sum_{k=0}^{1}c_{1,k}=1$ and $c_{0,0}=1$, 
it follows from \cref{e:3.11} that 
$\sum_{k=0}^2c_{2,k}=1$ (which can also be seen from \cref{e:c2k}).
Now assume that \cref{p:3.1i} 
holds for all $n\leq m$, for some fixed $m\geq 2$. 
Then \cref{e:3.11} implies,
        \begin{align*}
            \sum_{k=0}^{m+1}c_{m+1,k}&=\frac{2(m+1)+\alpha}{(m+1)+\alpha}
            \sum_{k=0}^{m}c_{m,k}-\frac{m+1}{(m+1)+\alpha}\sum_{k=0}^{m-1}c_{m-1, k}\\ 
&=\frac{2(m+1)+\alpha}{(m+1)+\alpha}-\frac{m+1}{(m+1)+\alpha}=1,
        \end{align*} 
and we have shown \cref{p:3.1i} by induction.

\cref{p:3.1ii}:
The result is clear when $n=0$, or when $k\in\{0,n\}$ (by \cref{r:cn0}). 
Now we set 
\begin{equation}
\label{e:defdelta}
(\forall n\geq 1)(\forall k\in \mathbb{Z})\quad 
\delta_{n,k}:=c_{n,k}-\frac{2n+\alpha}{2(n+\alpha)}c_{n-1,k}.
\end{equation}
We compute 
\begin{align*}
\delta_{1,0}&=c_{1,0}-\frac{2+\alpha}{2(1+\alpha)}c_{0,0}
=0, \tag{by \cref{e:cnk01}}\\ 
\delta_{1,1}&=c_{1,1}=\frac{\alpha}{2(1+\alpha)} > 0\tag{by \cref{e:cnk01}}\\ 
\delta_{1,k}&=0, \quad \text{for all $k\in \mathbb{Z}\smallsetminus \{0,1\}$.} 
\tag{because $c_{n,k} = 0$} 
\end{align*}
Therefore, $\delta_{1,k}$ and $c_{1,k}$ are nonnegative for all 
$k\in \mathbb{Z}$, and  
$\delta_{1,1},c_{1,0},c_{1,1}$ are all positive. 
Now assume that for some $n\geq 1$, the numbers 
$\delta_{n,k}$ and $c_{n,k}$ are nonnegative 
for all $k\in \mathbb{Z}$ and $c_{n,0},\ldots,c_{n,n}$ are all positive. 
Then, using \cref{e:defdelta} and \cref{e:cnk}, we obtain
\begin{subequations}
\label{orange}
\begin{align}
\delta_{n+1,k}
&= 
c_{n+1,k}-\frac{2(n+1)+\alpha}{2(n+1+\alpha)}c_{n,k}
\label{orange1}\\
&=
\frac{2(n+1)+\alpha}{2(n+1+\alpha)}c_{n,k}
   + \frac{2(n+1)+\alpha}{2(n+1+\alpha)}c_{n,k-1}
   -\frac{n+1}{n+1+\alpha}c_{n-1,k-1} 
   -\frac{2(n+1)+\alpha}{2(n+1+\alpha)}c_{n,k}
   \label{orange2}\\
&=
   \frac{2(n+1)+\alpha}{2(n+1+\alpha)}c_{n,k-1}
   -\frac{n+1}{n+1+\alpha}c_{n-1,k-1} 
   \\
&= \frac{2(n+1)+\alpha}{2(n+1+\alpha)}
\left(
c_{n,k-1} - \frac{2(n+1)}{2(n+1)+\alpha}c_{n-1,k-1}
\right)
\end{align}
\end{subequations}
On the other hand, 
\begin{align*}
c_{n,k-1}- \frac{2(n+1)}{2(n+1)+\alpha}c_{n-1,k-1} - \delta_{n,k-1}
&= 
-\frac{2(n+1)}{2(n+1)+\alpha}c_{n-1,k-1} 
+ \frac{2n+\alpha}{2(n+\alpha)}c_{n-1,k-1}\\
&= \frac{\alpha(\alpha-2)}{2(n+\alpha)(2(n+1)+\alpha)}c_{n-1,k-1}
\end{align*}
which re-arranges to 
\begin{equation}
\label{e:260325b}
c_{n,k-1}- \frac{2(n+1)}{2(n+1)+\alpha}c_{n-1,k-1}
= \delta_{n,k-1} + 
\frac{\alpha(\alpha-2)}{2(n+\alpha)(2(n+1)+\alpha)}c_{n-1,k-1}. 
\end{equation}
Substituting \cref{e:260325b} into \cref{orange} thus yields that
\begin{equation}
\delta_{n+1,k}
= 
c_{n+1,k}-\frac{2(n+1)+\alpha}{2(n+1+\alpha)}c_{n,k}
=
\frac{2(n+1)+\alpha}{2(n+1+\alpha)}
\left(
\delta_{n,k-1} + 
\frac{\alpha(\alpha-2)}{2(n+\alpha)(2(n+1)+\alpha)}c_{n-1,k-1}
\right)
\end{equation}
is nonnegative, and even  positive when $k\in\{1,\ldots,n\}$ 
(because then $c_{n-1,k-1}>0$). 
Consequently, 
\begin{equation}
(\forall k\in\{1,\ldots,n\})\quad
c_{n+1,k} > \frac{2(n+1)+\alpha}{2(n+1+\alpha)}c_{n,k} > 0.
\end{equation}
which completes the proof of the induction. 

\cref{p:3.1iii}: 
Set 
\begin{align*}
(\forall n\in \NN)(\forall k\in \mathbb{Z})\quad 
r_{n,k}:=\frac{c_{n,k}}{c_{n,0}}
\end{align*}
which is well-defined by \cref{e:260325a}. 
We now proceed along the following steps.
        
\textbf{Step 1:} Show that $(c_{n,0})_\nnn$ converges to $0$.\\
By \cref{e:260325a}, we have for $n\geq 1$ that 
        \begin{align*}
            \ln (c_{n,0})=\sum_{m=1}^n \ln \left(\frac{2m+\alpha}{2(m+\alpha)}\right)
            =\sum_{m=1}^n\ln\left(1-\frac{\alpha}{2(m+\alpha)}\right) 
            \leq -\sum_{m=1}^n \frac{\alpha}{2(m+\alpha)}\to -\infty. 
        \end{align*}
This implies $c_{n,0}\to 0$.

        \textbf{Step 2:} Show that for each $k\in \NN$, $(r_{n,k})_\nnn$ converges.\\
Let $n\geq 2$. Dividing \cref{e:cnk} on both sides by $c_{n,0}$ 
followed by using \cref{e:rcn0}, we obtain
        \begin{align*}
            \frac{c_{n,k}}{c_{n,0}}&=\frac{2n+\alpha}{2(n+\alpha)}
            \frac{c_{n-1,k}}{c_{n,0}}+\frac{2n+\alpha}{2(n+\alpha)}\frac{c_{n-1,k-1}}{c_{n,0}}
            -\frac{n}{n+\alpha}\frac{c_{n-2,k-1}}{c_{n,0}}\\ 
            &=\frac{2n+\alpha}{2(n+\alpha)}\frac{c_{n-1,k}}{c_{n-1,0}}\frac{c_{n-1,0}}{c_{n,0}}
            +\frac{2n+\alpha}{2(n+\alpha)}\frac{c_{n-1,k-1}}{c_{n-1,0}}\frac{c_{n-1,0}}{c_{n,0}}
            -\frac{n}{n+\alpha}\frac{c_{n-2,k-1}}{c_{n-2,0}}\frac{c_{n-2,0}}{c_{n,0}}\\ 
            &=\frac{2n+\alpha}{2(n+\alpha)}\frac{2(n+\alpha)}{2n+\alpha}\frac{c_{n-1,k}}{c_{n-1,0}}
            +\frac{2n+\alpha}{2(n+\alpha)}\frac{2(n+\alpha)}{2n+\alpha}\frac{c_{n-1,k-1}}{c_{n-1,0}} 
            -\frac{n}{n+\alpha}\frac{2(n+\alpha)}{2n+\alpha}\frac{2(n-1+\alpha)}{2(n-1)+\alpha}\frac{c_{n-2,k-1}}{c_{n-2,0}}\\ 
            &=\frac{c_{n-1,k}}{c_{n-1,0}}+\frac{c_{n-1,k-1}}{c_{n-1,0}}-\frac{4n(n+\alpha-1)}{(2n+\alpha)(2n+\alpha-2)}\frac{c_{n-2,k-1}}{c_{n-2,0}}.
        \end{align*}
This yields
        \begin{align*}
            r_{n,k}=r_{n-1,k}+r_{n-1,k-1}-\frac{4n(n+\alpha-1)}{(2n+\alpha)(2n+\alpha-2)}r_{n-2,k-1};
        \end{align*}
equivalently,
\begin{subequations}
\label{e:r iterate}
\begin{align}
r_{n,k}-r_{n-1,k}&=r_{n-1,k-1}-r_{n-2,k-1}+\left(1-\frac{4n(n+\alpha-1)}{(2n+\alpha)(2n+\alpha-2)}\right)r_{n-2,k-1}\\ 
&=r_{n-1,k-1}-r_{n-2,k-1}+\frac{\alpha(\alpha-2)}{(2n+\alpha)(2n+\alpha-2)}r_{n-2,k-1}.
\end{align}
\end{subequations}
        We claim that 
        \begin{equation}
            \label{e:claim}
            (\forall k\in \NN) (\exists \beta_k>0) (\forall n\geq 1) \ |r_{n,k}-r_{n-1,k}|\leq \frac{\beta_k}{n^2}. 
        \end{equation}
We shall prove this by induction on $k$. 
For $k=0$, \cref{e:claim} is obvious because $r_{n,0}=1$ for all $n\in \NN$. 
        Assume that \cref{e:claim} holds for some $k\in \NN$. Then 
        $|r_{n,k}-r_{n-1,k}|\leq \tfrac{\beta_k}{n^2}$ and thus $(r_{n,k})_\nnn$ converges to a limit
        $R_k\in \mathbb{R}$.
        Furthermore, \cref{e:r iterate} yields that for $n\geq 2$,
        \begin{align*}
            |r_{n,k+1}-r_{n-1, k+1}|&\leq |r_{n-1,k}-r_{n-2,k}|+\frac{\alpha(\alpha-2)}{(2n+\alpha)(2n+\alpha-2)}|r_{n-2,k}|\\ 
            &\leq \frac{\beta_k}{(n-1)^2}+\frac{\alpha(\alpha-2)}{(2n+\alpha)(2n+\alpha-2)}|r_{n-2,k}|.
        \end{align*}
        This implies 
        \begin{align*}
            n^2|r_{n,k+1}-r_{n-1,k+1}|\leq 
\big(1-\tfrac{1}{n} \big)^{-2}\beta_k+\frac{\alpha(\alpha-2)n^2}{(2n+\alpha)(2n+\alpha-2)}|r_{n-2,k}|.
        \end{align*}
        Therefore, 
        \begin{align*}
            \varlimsup_{n\to \infty}n^2|r_{n,k+1}-r_{n-1,k+1}|\leq \beta_k+\frac{\alpha(\alpha-2)}{4}|R_k|.
        \end{align*}
Hence there exists $\beta_{k+1}>0$ such that 
\begin{align*}
(\forall n\geq 1)\quad  n^2|r_{n,k+1}-r_{n-1,k+1}|\leq \beta_{k+1};
\end{align*}
equivalently,
\begin{align*}
(\forall n\geq 1) \quad 
|r_{n,k+1}-r_{n-1,k+1}|\leq \frac{\beta_{k+1}}{n^2}. 
\end{align*}
This completes the induction argument and thus \cref{e:claim} is true. 
In turn,  for every $k\in \NN$, $(r_{n,k})_\nnn$ is therefore 
a Cauchy sequence and hence convergent (with a finite-length orbit). 

        \textbf{Step 3:} Show that for each $k\in \NN$, $(c_{n,k})_\nnn$ converges to $0$.\\
Fix $k\in \NN$. By \textbf{Step~2}, there exists $R_k\in \RR$ such that 
$r_{n,k}=c_{n,k}/c_{n,0}\to R_k$ as $n\to\infty$. 
On the other hand, by \textbf{Step~1}, $c_{n,0}\to 0$ as $n\to\infty$. 
Altogether, 
$c_{n,k} = (c_{n,k}/c_{n,0})c_{n,0}\to R_k\cdot 0 = 0$ as 
$n\to\infty$.
\end{proof}

We are now ready for a main result.
    
\begin{theorem}[regularity of the coefficient matrix and identification of the weak limit]
\label{p:3.2}
The coefficient matrix 
$(c_{n,k})_{(n,k)\in\NN\times\NN}\in\RR^{\NN\times\NN}$ 
from  \cref{f:2.2} is regular, and the
sequence $(x_n)_\nnn$ generated by the \BN\ algorithm (see \cref{e:fastKM})
satisfies 
\begin{equation}
x_n \weakly P_{\Fix T}x_0. 
\end{equation}
\end{theorem}
\begin{proof}
\cref{p:3.1} implies that $C$ satisfies 
the Silverman-Toeplitz conditions; therefore, 
the regularity follows from \cref{f:ST}. 

It was shown by Bo\c{t} and Nguyen in 
\cite[Theorem~3.4]{BN} that $(x_n)_\nnn$ 
converges weakly to a fixed point of $T$. 
The precise identification of the weak limit as $P_{\Fix T}x_0$ 
follows by combining \cref{p:3.1} with \cref{Cohenlim}. 
\end{proof}

\begin{remark} 
The results in this section came as a big surprise to us 
--- we didn't expect to encounter a weighted mean ergodic iteration with positive coefficients! 
Let us list the first few polynomials\footnote{We used \texttt{SageMath 10.6} to
aid in this computation.} from \cref{e:fastKMc}: 
\begin{align*}
p_0(t) &= 1,\\
p_1(t) &= \frac{\alpha + 2}{2(\alpha + 1)} + \frac{\alpha}{2(\alpha + 1)} t,\\
p_2(t) &= \frac{\alpha + 4}{4(\alpha + 1)} + \frac{\alpha}{2(\alpha + 2)} t 
+ \frac{\alpha(\alpha + 4)}{4(\alpha + 1)(\alpha + 2)} t^2,\\ 
p_3(t) &= 
\frac{(\alpha+4)(\alpha+6)}{8(\alpha+1)(\alpha+3)} + 
\frac{\alpha(\alpha+4)(3\alpha+2)}{8(\alpha+1)(\alpha+2)(\alpha+3)} t 
+
\frac{3\alpha(\alpha+2)}{8(\alpha+1)(\alpha+3)} t^2
+
\frac{\alpha(\alpha+4)(\alpha+6)}{8(\alpha+1)(\alpha+2)(\alpha+3)} t^3.
\end{align*}
It is clear from this representation that the coefficients 
are all positive which triggered the analysis in this section. 
The numerators, however, get 
increasingly complicated and cannot be factored 
into linear terms featuring only integer coefficients. 
\end{remark}

\begin{remark}[$p_n$ vs $q_n$]
Consider the recursive formulas for $p_n$ and $q_n$ 
given by \cref{e:pnrecn} by 
\cref{e:qrecursion}, respectively.
Up to a shift in the index and up to 
the initial conditions, they match perfectly provided that $\beta = \alpha/2$! 
This hints at the importance of the beta-binomial distribution 
in the analysis of the \BN\ algorithm.
In the next section, we will confirm this connection 
rigorously for the special case when $\alpha=4$ and hence $\beta=2$.
\end{remark}

\section{The Case when $\alpha=4$ (and $\beta=2$)}

\label{s:alpha4}

In this section, we assume throughout that
\begin{empheq}[box=\mybluebox]{equation}
\alpha =4 \quad \text{and} \quad \beta = 2. 
\end{empheq}

\subsection*{Guessing a Formula for $c_{n,k}$}

Using \cref{r:pnrecursion}, we compute 
\begin{subequations}
\label{e:pnalpha4}
\begin{align}
p_0(t)&=1,
\\[1mm]
p_1(t)&=\frac35+\frac25\,t,
\\[1mm]
p_2(t)&=\frac25+\frac13\,t+\frac4{15}\,t^2,
\\[1mm]
p_3(t)&=\frac27+\frac4{15}\,t+\frac9{35}\,t^2+\frac4{21}\,t^3,
\\[1mm]
p_4(t)&=\frac3{14}+\frac3{14}\,t+\frac{19}{84}\,t^2+\frac{17}{84}\,t^3+\frac17\,t^4,
\\[1mm]
p_5(t)&=\frac16+\frac{11}{63}\,t+\frac7{36}\,t^2+\frac4{21}\,t^3+\frac{41}{252}\,t^4+\frac19\,t^5,
\\[1mm]
p_6(t)&=\frac2{15}+\frac{13}{90}\,t+\frac16\,t^2+\frac{31}{180}\,t^3+\frac{29}{180}\,t^4+\frac2{15}\,t^5+\frac4{45}\,t^6,
\end{align}
\end{subequations}
which is not particularly informative. 
However, considering 
\begin{equation}
D_n := (n+2)(n+3)(n+4), 
\end{equation} 
and then $D_np_n(t)$ reveals  
\begin{subequations}
\label{e:260403a}
\begin{align}
D_0p_0(t)&=24,\\
D_1p_1(t)&=36+24t,
\\
D_2p_2(t)&=48+40t+32t^2,
\\
D_3p_3(t)&=60+56t+54t^2+40t^3,
\\
D_4p_4(t)&=72+72t+76t^2+68t^3+48t^4,
\\
D_5p_5(t)&=84+88t+98t^2+96t^3+82t^4+56t^5,
\\
D_6p_6(t)&=96+104t+120t^2+124t^3+116t^4+96t^5+64t^6.
\end{align}
\end{subequations}
We spot that the constant coefficient in $D_np_n(t)$ is $12(n+2)$, 
and so we guess the formula 
\begin{equation}
c_{n,0} = \frac{12(n+2)}{D_n} = \frac{12}{(n+3)(n+4)}. 
\end{equation}
Recalling \cref{e:cnn1}, we guess that for $n\geq 1$, we have 
\begin{equation}
c_{n,n} = \frac{4}{2+4}c_{n,0} = \frac{2}{3} \cdot \frac{12}{(n+3)(n+4)}
= \frac{8}{(n+3)(n+4)}.
\end{equation}
For fixed $n$, the $k$th coefficient of 
the polynomial $D_n p_n(t)$, where $k\ge 1$, appears to be quadratic in $k$: 
indeed, when $n=5$ the row coefficients
\[
88,98,96,82,56
\]
has first differences
$10,-2,-14,-26$ 
and second differences
$-12,-12,-12$. 
So for fixed $n$, we guess the quadratic (in $k$) formula: 
\[
D_n c_{n,k}=\alpha_n k^2+\beta_n k+\gamma_n, 
\]
where $k\in\{1,2,\ldots,n\}$. 
On the other hand, for fixed $k$, the passage from $n$ to $n+1$ appears to be linear in $n$: 
for instance, when $k=2$, we encounter the column coefficients 
\[32,54,76,98,120.\]
Investigating
a linear ansatz for the coefficients $\alpha_n, \beta_n, \gamma_n$,
we are eventually led to the following conjecture: 
\begin{equation}
(\forall 1\leq k\leq n)\quad
D_n c_{n,k}=-6k^2+(6n-2)k+(10n+16).
\end{equation}

Our next step is to verify this conjecture.

\subsection*{Verification of the Formula for $c_{n,k}$}

\begin{proposition}
The (nonzero) coefficients $c_{n,k}$ are given by
\begin{subequations}
\label{e:260402a}
\begin{equation}
(\forall\nnn)\quad 
c_{n,0}=\frac{12}{(n+3)(n+4)},
\end{equation}
and 
\begin{equation}
(\forall n\geq 1)(\forall k\in \{1,\ldots,n\})\quad
c_{n,k}=
\frac{2\bigl(-3k^2+(3n-1)k+5n+8\bigr)}{(n+2)(n+3)(n+4)};
\end{equation}
\end{subequations}
equivalently, 
\begin{equation}
\label{e:260402b}
p_0(t) = 1, \quad \text{and} \quad
(\forall n\geq 1)\quad p_n(t)=\frac{12}{(n+3)(n+4)}+
\sum_{k=1}^n \frac{2\bigl(-3k^2+(3n-1)k+5n+8\bigr)}{(n+2)(n+3)(n+4)}\,t^k.
\end{equation}
\end{proposition}
\begin{proof}
First, it's clear that \cref{e:260402a} is the same as \cref{e:260402b}, 
we'll focus on the former. 
For $n\geq 1$ and $k\in\{1,\ldots,n\}$, abbreviate 
\begin{equation}
\label{e:defQn}
Q_n(k):=-3k^2+(3n-1)k+5n+8.
\end{equation}
Note that $Q_n(k)$ nicely simplifies when $k\in\{1,n\}$: 
\begin{equation}
\label{e:Qnn}
Q_n(1) = 4(2n+1) 
\quad\text{and}\quad 
Q_n(n) = 4(n+2). 
\end{equation}
Our goal is to prove that 
\begin{equation}
\label{e:ziel}
c_{n,0}=\frac{12}{(n+3)(n+4)}
\quad \text{and} \quad
(\forall k\in\{1,\ldots,n\}) \quad 
c_{n,k}=\frac{2Q_n(k)}{(n+2)(n+3)(n+4)}. 
\end{equation}
We prove this by (strong) induction on $n$. 
The cases $n=1$ and $n=2$ are verified by a direct computation 
(use \cref{e:pnalpha4}) which we omit. 

Now assume that \cref{e:ziel} holds for $n-1$ and $n$, where $n\geq 2$. 
Let $k\in\{0,1,\ldots,n+1\}$. We argue by cases. 

\emph{Case~1:} $k=0$.\\ 
Using \cref{e:rcn0} (with $n+1$ and $\alpha=4$) and the inductive 
hypothesis, we have
\begin{align*}
c_{n+1,0} &= \frac{2(n+1)+4}{2(n+1+4)}\cdot c_{n,0}
= \frac{2(n+3)}{2(n+1+4)}\cdot \frac{12}{(n+3)(n+4)}
= \frac{12}{(n+4)(n+5)},
\end{align*}
as claimed. 

\emph{Case~2:} $k=1$. 
Using \cref{e:cnk} (with $n+1$ and $\alpha=4$) and \cref{e:Qnn},
\begin{align*}
c_{n+1,1} 
&= 
\frac{2(n+1)+4}{2(n+1+4)}c_{n,1}
+\frac{2(n+1)+4}{2(n+1+4)}c_{n,0}-\frac{n+1}{n+1+4}c_{n-1,0}\\
&=
\frac{n+3}{n+5}\cdot\frac{2Q_n(1)}{(n+2)(n+3)(n+4)}
+\frac{n+3}{n+5}\cdot\frac{12}{(n+3)(n+4)}-\frac{n+1}{n+5}\cdot \frac{12}{(n+2)(n+3)}\\
&= 
\frac{8(2n+1)}{(n+2)(n+4)(n+5)}
+\frac{12}{(n+4)(n+5)}-\frac{12(n+1)}{(n+2)(n+3)(n+5)}\\
&= 
\frac{8(2n+3)}{(n+3)(n+4)(n+5)}\\
&=
\frac{2Q_{n+1}(1)}{(n+3)(n+4)(n+5)},
\end{align*}
as needed. 

\emph{Case~3:} $k=n+1$.\\ 
Using \cref{e:cnn1}, \emph{Case~1}, and \cref{e:Qnn}, we have 
\begin{align*}
c_{n+1,n+1}
&= \frac{4}{2+4}c_{n+1,0}
= \frac{2}{3}\cdot \frac{12}{(n+4)(n+5)}
= \frac{8}{(n+4)(n+5)}
= \frac{2\cdot4(n+3)}{(n+3)(n+4)(n+5)} 
\\
&= \frac{2Q_{n+1}(n+1)}{(n+3)(n+4)(n+5)},
\end{align*}
which yields \cref{e:ziel} for $k=n+1$.

\emph{Case~4:} $k\in\{2,\ldots,n\}$.\\ 
The inductive hypothesis yields that 
\begin{equation*}
c_{n,k}=\frac{2Q_n(k)}{(n+2)(n+3)(n+4)},
\;\;
c_{n,k-1}=\frac{2Q_n(k-1)}{(n+2)(n+3)(n+4)},
\;\;
c_{n-1,k-1}=\frac{2Q_{n-1}(k-1)}{(n+1)(n+2)(n+3)}.
\end{equation*}
Substituting these into \cref{e:cnk} (with $n+1$ and $\alpha=4$) 
followed by some algebraic manipulations, we obtain 
\begin{equation}
\label{e:momcal1}
c_{n+1,k} = 
\frac{2}{(n+2)(n+3)(n+4)(n+5)}
\Bigl((n+3)\bigl(Q_n(k)+Q_n(k-1)\bigr)-(n+4)Q_{n-1}(k-1)\Bigr). 
\end{equation}
On the other hand, by \cref{e:defQn}, we have 
(after some algebraic manipulations) that
\begin{equation}
\label{e:momcal2}
(n+3)\bigl(Q_n(k)+Q_n(k-1)\bigr)-(n+4)Q_{n-1}(k-1)
= (n+2)Q_{n+1}(k).
\end{equation}
Substituting \cref{e:momcal2} into \cref{e:momcal1}, we obtain 
\begin{equation*}
c_{n+1,k}=
\frac{2(n+2)Q_{n+1}(k)}{(n+2)(n+3)(n+4)(n+5)}
=\frac{2Q_{n+1}(k)}{(n+3)(n+4)(n+5)},
\end{equation*}
as required.
\end{proof}

\subsection*{A Formula for the Forward Differences of $c_{n,k}$}

We now define the forward differences of $c_{n,k}$ by 
\begin{empheq}[box=\mybluebox]{equation}
(\forall\nnn)(\forall k\in\ZZ)\quad
d_{n,k} := c_{n,k+1}-c_{n,k}. 
\end{empheq}

\begin{proposition}[Formula for the forward differences] 
We have $d_{0,-1}=1$, $d_{0,0}=-1$, and $d_{0,k}=0$ for $k\in \ZZ\setminus\{-1,0\}$, 
and 
\begin{align}
\label{e:fodi}
(\forall n\geq 1)(\forall k\in \ZZ)\quad 
d_{n,k} &= 
\begin{cases}
\displaystyle \frac{12}{(n+3)(n+4)}, & \text{if $k=-1$;}\\[+5mm]
\displaystyle \frac{4(n-4)}{(n+2)(n+3)(n+4)}, & \text{if $k=0$;}\\[+5mm]
\displaystyle \frac{2(3n-6k-4)}{(n+2)(n+3)(n+4)}, 
& \text{if $1\leq k\leq n-1$;}\\[+5mm]
\displaystyle \frac{-8}{(n+3)(n+4)}, & \text{if $k=n$;}\\[+5mm]
\displaystyle 0, & \text{otherwise.}
\end{cases}
\end{align}
\end{proposition}
\begin{proof}
The formula for $d_{0,k}$ is clear. 
Now let $n\geq 1$ and $k\in \ZZ$. 
It is also clear that $d_{n,k}=0$ for $k\in \ZZ\smallsetminus\{-1,0,\ldots,n\}$. 
So let $k\in \{-1,0,1,\ldots,n\}$. 
We argue by cases and use \cref{e:260402a} each time:

\emph{Case~1:} $k=-1$.\\
Then 
\begin{align*}
d_{n,-1} &= c_{n,0}-c_{n,-1}
= \frac{12}{(n+3)(n+4)}-0
= \frac{12}{(n+3)(n+4)}.
\end{align*}

\emph{Case~2:} $k=0$.\\
Then 
\begin{align*}
d_{n,0} &= c_{n,1}-c_{n,0}
= \frac{8(2n+1)}{(n+2)(n+3)(n+4)}-\frac{12}{(n+3)(n+4)}
= \frac{4(n-4)}{(n+2)(n+3)(n+4)}.
\end{align*}

\emph{Case~3:} $k\in\{1,\ldots,n-1\}$.\\
Then 
\begin{align*}
d_{n,k} &= c_{n,k+1}-c_{n,k}
= \frac{2\big(-3(k+1)^2 + (3n-1)(k+1)+5n+8\big)}{(n+2)(n+3)(n+4)}
-
\frac{2\big(-3k^2 + (3n-1)k+5n+8\big)}{(n+2)(n+3)(n+4)}\\
&=\frac{2(3n-6k-4)}{(n+2)(n+3)(n+4)}.
\end{align*}

\emph{Case~4:} $k=n$.\\
Then 
\begin{align*}
d_{n,n} &= c_{n,n+1}-c_{n,n}
= 0-\frac{8}{(n+3)(n+4)}
= \frac{-8}{(n+3)(n+4)}.
\end{align*}
The proof is complete.
\end{proof}

\subsection*{Unimodality of $(c_{n,k})_{k\in\NN}$}

\begin{proposition}[unimodality of $(c_{n,k})_{k\in\NN}$]
\label{p:unimod} 
For $n\geq 4$, 
the sequence $(c_{n,k})_{k\in\NN}$ is unimodal. 
In fact, 
if $n\leq 3$, then
$c_{n,0} > c_{n,1}> \cdots > c_{n,n} > 0 = c_{n,n+1}= c_{n,n+2}=\cdots$; 
if $n=4$, then 
$c_{4,0} = c_{4,1} < c_{4,2} > c_{4,3} > c_{4,4} > 0 = c_{4,5}=\cdots$;
while 
\begin{equation}
\label{e:unimod}
c_{n,0} < c_{n,1} < \cdots < c_{n,\lfloor n/2\rfloor} 
> c_{n,\lfloor n/2\rfloor+1} > \cdots > c_{n,n} > 0 = c_{n,n+1}=\cdots
\quad\text{when $n\geq 5$.}
\end{equation}
\end{proposition}
\begin{proof}
The statement for $n\leq 4$ is clear by inspecting \cref{e:260403a}.  

Now assume that $n\geq 5$. 
By \cref{e:fodi}, $d_{n,0}>0$, $d_{n,n}<0$, 
and $d_{n,k}=0$ for $k\geq n+1$. 
Set $m_n := \lfloor n/2\rfloor$, 
and let $k\in\{1,\ldots,n-1\}$. 
By \cref{e:fodi}, the sign of $d_{n,k}$ is the same as the sign of 
$3n-6k-4$. To complete the proof, we argue by cases.

\emph{Case~1:} $n$ is even.\\
Then $m_n= n/2$ and hence $3n-6k-4= 6(m_n-k)-4$. 
The last expression is positive if $k\leq m_n-1$ and negative if $k\geq m_n$.

\emph{Case~2:} $n$ is odd.\\
Then $m_n= (n-1)/2$ and hence 
$3n-6k-4= 6(m_n-k)-1$.
Again, the last expression is positive if 
$k\leq m_n-1$ and negative if $k\geq m_n$.
\end{proof}

\begin{corollary}
\label{c:unimod}
Let $n\geq 5$.
Then
\begin{equation}
\sum_{k=0}^\infty |d_{n,k}| = 
\begin{cases}
\displaystyle \frac{3n^2+6n+8}{(n+2)(n+3)(n+4)}, & \text{if $n$ is even;}\\[+5mm]
\displaystyle \frac{3n^2+6n+7}{(n+2)(n+3)(n+4)}, & \text{if $n$ is odd.}
\end{cases}
\end{equation}
\end{corollary}
\begin{proof}
Set $m_n := \lfloor n/2\rfloor$.  
Using \cref{e:unimod}, we have
\begin{subequations}
\begin{align}
\Delta_n &:= \sum_{k=0}^\infty |d_{n,k}| 
= \sum_{k=0}^n |d_{n,k}| 
= \sum_{k=0}^{m_n-1} d_{n,k} 
+ \sum_{k=m_n}^n (-d_{n,k}) 
=c_{n,m_n}-c_{n,0}+c_{n,m_n}-c_{n,n+1} \\
&= 2c_{n,m_n} - c_{n,0}.
\end{align} 
\end{subequations}
We now argue by cases. 

\emph{Case~1:} $n$ is even.\\
Then $m_n = n/2$ and hence \cref{e:260402a} yields 
\begin{align*}
\Delta_n &= 2c_{n,n/2} - c_{n,0}
= \frac{4\bigl(-3(n/2)^2+(3n-1)(n/2)+5n+8\bigr)}{(n+2)(n+3)(n+4)} 
- \frac{12}{(n+3)(n+4)}\\
&= \frac{3n^2+6n+8}{(n+2)(n+3)(n+4)}. 
\end{align*}

\emph{Case~2:} $n$ is odd.\\
Then $m_n = (n-1)/2$ and hence \cref{e:260402a} yields 
\begin{align*}
\Delta_n &= 2c_{n,(n-1)/2} - c_{n,0}
= \frac{4\bigl(-3((n-1)/2)^2+(3n-1)((n-1)/2)+5n+8\bigr)}{(n+2)(n+3)(n+4)} 
- \frac{12}{(n+3)(n+4)}\\
&= \frac{3n^2+6n+7}{(n+2)(n+3)(n+4)}. 
\end{align*}
This completes the proof. 
\end{proof}

\subsection*{Lorentz Condition and  Strong Convergence of \BN\ Acceleration}

We are now in position to present another main result: 

\begin{theorem}[strong convergence of \BN\ acceleration when $\alpha=4$]
\label{t:super4}
Recall that $\alpha=4$. 
Then coefficient matrix 
$(c_{n,k})_{(n,k)\in\NN\times\NN}\in\RR^{\NN\times\NN}$ 
from  \cref{f:2.2} satisfies the Lorentz condition {\rm \textbf{(L)}}. 
Consequently, if $T\colon X\to X$ is nonexpansive and linear, then the 
sequence $(x_n)_\nnn$ generated by the \BN\ algorithm (see \cref{e:fastKM})
satisfies 
\begin{equation}
\label{e:super4}
x_n \to P_{\Fix T}x_0. 
\end{equation}
\end{theorem}
\begin{proof}
By \cref{c:unimod}, we have 
$\sum_{k=0}^\infty |c_{n,k}-c_{n,k+1}| \sim (3/n) \to 0$ as 
$n\to\infty$, and so 
the Lorentz condition {\rm \textbf{(L)}} holds.
The conclusion now follows from \cref{p:3.2} and \cref{cool}.
\end{proof}

Let us illustrate \cref{t:super4} with the following:

\begin{example}[right-shift operator in $\ell^2$]
Suppose that $X=\ell^2(\NN)$ and 
$T\colon X\to X$ is the right-shift operator, 
which has $\Fix T = \{0\}$. 
Suppose that $x_0 = (1,0,0,\ldots)\in X$. 
Using \cref{e:260402a} and some tedious computations which are omitted here, 
we can show that 
\begin{equation}
\|x_n\|^2\sim \frac{6}{5n},
\quad
\|x_n-x_{n+1}\|^2 \sim \frac{24}{5n^3},
\quad
\|x_n-Tx_n\|^2 \sim \frac{12}{n^3},
\end{equation}
which illustrates not only the strong convergence $\|x_n-0\|\to 0$, with rate $\Theta(1/\sqrt{n})$ from \cref{e:super4} but also the fast convergence rate $o(1/n)$
(in fact, $\Theta(1/n^{3/2})$) for the discrete velocity and fixed point residual from \cref{BNaccel}. 
In contrast, for the classical (Cesaro equal-weights) mean ergodic theorem,
one computes 
\begin{equation}
\|x_n\|^2\sim \frac{1}{n},
\quad
\|x_n-x_{n+1}\|^2 \sim \frac{1}{n^2},
\quad
\|x_n-Tx_n\|^2 \sim \frac{2}{n^2},
\end{equation}
which yields a slightly faster strong convergence of the iterates but 
a worse convergence rate of $\Theta(1/n)$ for the discrete velocity and 
fixed point residual. 
\end{example}

\subsection*{Comparing $c_{n,k}$ to $b_{n,k}$}

Recall that $\beta=2$ by assumption in this section. 
We now revisit the coefficients $b_{n,k}$ defined by \cref{e:bnk} and
present a formula for them.

\begin{proposition}
We have 
\begin{equation}
\label{e:bnk2}
(\forall\nnn)(\forall k\in\{0,1,\ldots,n\})\quad 
b_{n,k}=\frac{6(k+1)(n-k+1)}{(n+1)(n+2)(n+3)}.
\end{equation}
\end{proposition}
\begin{proof}
It follows from \cref{f:BG}\cref{f:BG2}\&\cref{f:BG4} that
\begin{equation}
\label{e:B22}
\Bf(2,2)=\frac{\Gamma(2)\Gamma(2)}{\Gamma(4)}=
\frac{1!1!}{3!}=\frac16
\end{equation}
and
\begin{equation*}
\Bf(k+2,n-k+2)=\frac{\Gamma(k+2)\Gamma(n-k+2)}{\Gamma(n+4)}
=\frac{(k+1)!(n-k+1)!}{(n+3)!}. 
\end{equation*}
Hence 
\begin{align*}
b_{n,k}
&=\binom{n}{k}\frac{\Bf(k+2,n-k+2)}{\Bf(2,2)}
=6\binom{n}{k}\frac{(k+1)!(n-k+1)!}{(n+3)!}
=6\frac{n!}{k!(n-k)!}\frac{(k+1)!(n-k+1)!}{(n+3)!}
\\
&=\frac{6(k+1)(n-k+1)}{(n+1)(n+2)(n+3)},
\end{align*}
as claimed.
\end{proof}

\begin{proposition}[Comparing $c_{n,k}$ to $b_{n,k}$]
For every $n\in \NN$ and 
$k\in \NN$, set 
\begin{equation}
\delta_{n,k} := c_{n,k}-b_{n,k}.
\end{equation}
Then 
\begin{equation}
(\forall n\in \NN)(\forall k\in \NN)\quad \delta_{n,k} = 
\begin{cases}
\displaystyle  \frac{6n}{(n+2)(n+3)(n+4)}, & \text{if $k=0$;}\\[+5mm]
\displaystyle \frac{2\bigl(9k^2-(10n+1)k+2n^2-2n-4\bigr)}{(n+1)(n+2)(n+3)(n+4)}, 
& \text{if $1\leq k\leq n$;}\\[+5mm]
0, &\text{if $n+1\leq k$.}
\end{cases}
\end{equation}
Moreover, 
\begin{equation}
\label{e:kabeljau5}
\max_{k\in \NN} |\delta_{n,k}| = \delta_{n,0} = \frac{6n}{(n+2)(n+3)(n+4)}.
\end{equation}
\end{proposition}
\begin{proof}
When $n+1\leq k$, then
it is clear that $\delta_{n,k}=0-0=0$.
Using \cref{e:260402a} and \cref{e:bnk2}, we have 
\begin{equation*}
\delta_{n,0}=\frac{12}{(n+3)(n+4)}-\frac{6}{(n+2)(n+3)}
=\frac{6n}{(n+2)(n+3)(n+4)},
\end{equation*}
and, for $k\in \{1,\ldots,n\}$,
\begin{align*}
\delta_{n,k} &= 
\frac{2\bigl(-3k^2+(3n-1)k+5n+8\bigr)}{(n+2)(n+3)(n+4)}
-
\frac{6(k+1)(n-k+1)}{(n+1)(n+2)(n+3)}\\
&=
\frac{2\bigl(9k^2-(10n+1)k+2n^2-2n-4\bigr)}{(n+1)(n+2)(n+3)(n+4)}.
\end{align*}
Fix $k\in\{1,\ldots,n\}$. 
If $n=1$, then $k=1$ and hence 
$\delta_{1,1} = -1/10$ and so 
$|\delta_{1,1}| = 1/10 = \delta_{1,0}$.
Thus, concerning the ``Moreover'' part, we may and do
assume that $n\geq 2$, and then it suffices to show that 
\begin{equation}
\left|9k^2-(10n+1)k+2n^2-2n-4\right|< 3n(n+1).
\end{equation}
Define 
\begin{equation}
f(x) := 9x^2-(10n+1)x+2n^2-2n-4,
\end{equation}
which is a convex quadratic function. 
Hence 
\begin{equation}
\label{e:kjau1}
f(k) \leq \max\{f(1),f(n)\} = 
\max\{2(n^2-6n+2),n^2-3n-4\} < 3n(n+1).
\end{equation}
The minimizer of $f$ is 
$(10n+1)/(18)$ and the minimum value is
\begin{equation}
\min f(\RR) = f\big((10n+1)/18\big) = - \frac{28n^2+92n+145}{36}
\end{equation}
It follows that 
\begin{equation}
f(k)+3n(n+1) \geq \min f(\RR) +3n(n+1) = \frac{80n^2+16n-145}{36} > 0
\end{equation}
because $n\geq 2$. 
Hence 
\begin{equation}
\label{e:kjau2}
-f(k)< 3n(n+1).
\end{equation}
Combining \cref{e:kjau1} and \cref{e:kjau2}, we obtain
$|f(k)|< 3n(n+1)$, and we're done. 
\end{proof}

\begin{remark}[Lorentz condition of $(c_{n,k})$ via $(b_{n,k})$]
\label{r:Lcviab}
We've used \cref{c:unimod} to verify the Lorentz condition of $(c_{n,k})$ 
in the proof of \cref{t:super4}.
An alternative route to this result via $(b_{n,k})$ is as follows:
Using \cref{e:kabeljau5}, \cref{l:b4} and \cref{e:B22}, we have 
\begin{align*}
\sum_{k=0}^\infty |c_{n,k}-c_{n,k+1}|
&\leq 
\sum_{k=0}^\infty |c_{n,k}-b_{n,k}| + \sum_{k=0}^\infty |b_{n,k}-b_{n,k+1}| + \sum_{k=0}^\infty |b_{n,k+1}-c_{n,k+1}|\\
&\leq 
2\sum_{k=0}^\infty |c_{n,k}-b_{n,k}| + \sum_{k=0}^\infty |b_{n,k}-b_{n,k+1}|\\
&\leq
\frac{12n(n+1)}{(n+2)(n+3)(n+4)} + \frac{4^{1-2}}{\Bf(2,2)}\cdot \frac{2}{n+1}\\
&=
\frac{12n(n+1)}{(n+2)(n+3)(n+4)} + \frac{3}{n+1}\\
&\to 0 \quad \text{as $n\to\infty$.}
\end{align*}
\end{remark}

\begin{remark}[visualizations]
We now illustrate some of the results in this section. 
In Figure~\ref{fig:panel}, 
we show the distributions for various values of $n$. 
It is clear that the beta-binomial dominates \BN\ in the center, 
that the \BN\ is asymmetric, and that the distributions get closer
as $n$ is increased. The difference between the distributions, 
for various value of $n$, is visualized in a rescaled heatmap in 
Figure~\ref{fig:heat}. 
It follows from \cref{e:kabeljau5} that 
$\sum_{k=0}^n|c_{n,k}-b_{n,k}|$ is $O(1/n)$. 
In fact, Figure~\ref{fig:loglog} suggests it is $\Theta(1/n)$.
\end{remark}

\begin{figure}[tbp]
  \centering
  \includegraphics[width=0.95\textwidth]{figures/section5_panel.tikz}
  \caption{Comparison of \(c_{n,k}\) and \(b_{n,k}\) for selected values of \(n\).}
\label{fig:panel}
\end{figure}

\begin{figure}[tbp]
\centering
\begin{tikzpicture}
\begin{axis}[
    width=0.88\textwidth,
    height=0.58\textwidth,
    scale only axis,
    xmin=0, xmax=1,
    ymin=1, ymax=220,
    enlargelimits=false,
    axis on top,
    xlabel={$k/n$},
    ylabel={$n$},
    xtick={0,0.2,0.4,0.6,0.8,1},
    ytick={1,50,100,150,200},
    tick label style={font=\small},
    label style={font=\normalsize},
    xticklabel style={/pgf/number format/fixed},
    yticklabel style={/pgf/number format/fixed},
    colorbar,
    point meta min=-5.8,
    point meta max=5.8,
    colormap={bwr}{
        rgb255(0cm)=(0,0,255);
        rgb255(1cm)=(255,255,255);
        rgb255(2cm)=(255,0,0)
    },
    colorbar style={
        ytick={-5.8,0,5.8},
        yticklabel style={/pgf/number format/fixed},
        tick label style={font=\small},
        title={$n^2(c_{n,k}-b_{n,k})$},
        title style={font=\small}
    }
]
\addplot graphics[
    xmin=0, xmax=1,
    ymin=1, ymax=220,
    includegraphics={trim={8pt 14pt 50pt 14pt},clip}
] {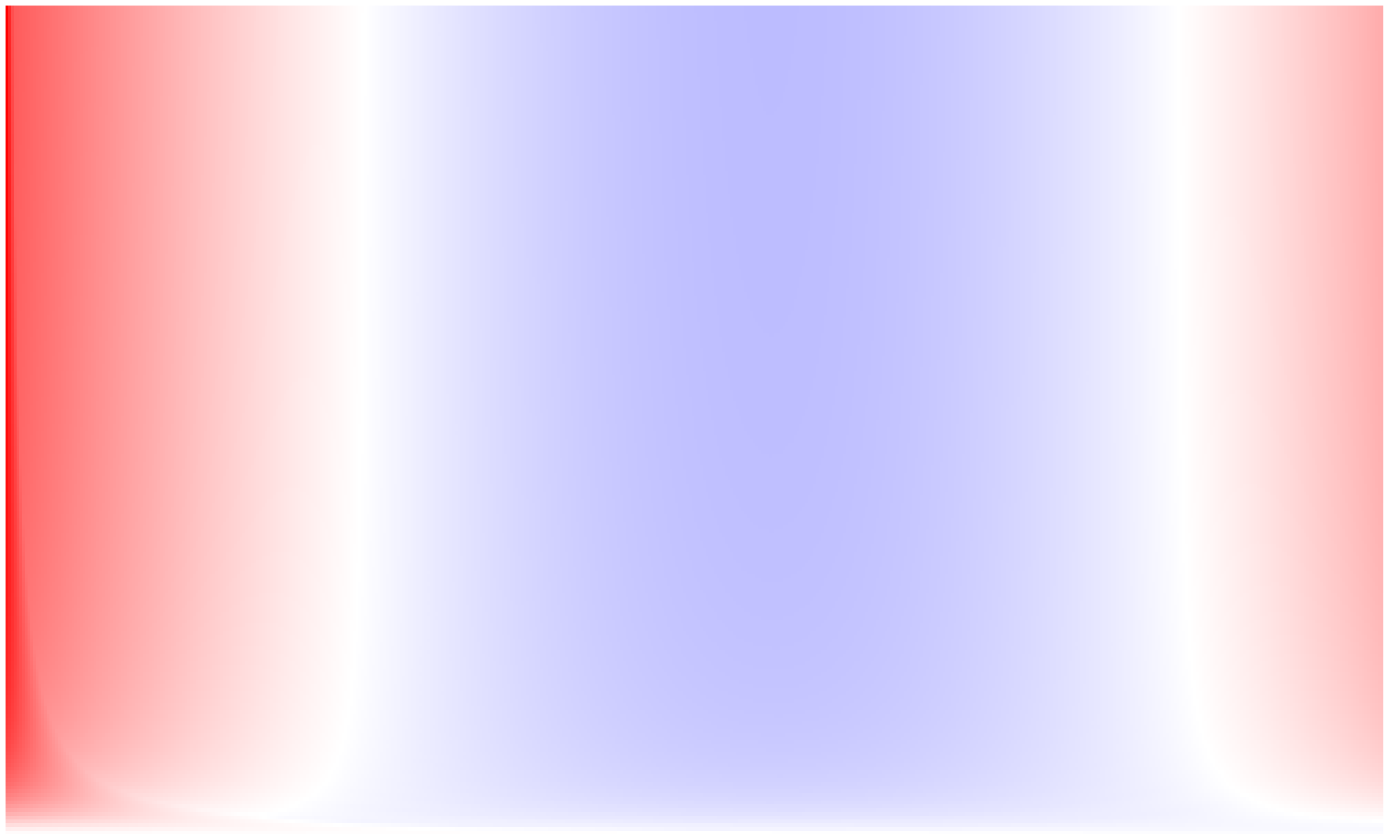};
\end{axis}
\end{tikzpicture}
\caption{Rescaled heatmap of $n^2(c_{n,k}-b_{n,k})$.}
\label{fig:heat}
\end{figure}

\begin{figure}[tbp]
  \centering
  \includegraphics[width=0.80\textwidth]{figures/section5_tv_loglog.tikz}
  \caption{Log-log plot of $\sum_{k=0}^n|c_{n,k}-b_{n,k}|$.}
  \label{fig:loglog}
\end{figure}

\section{Beyond $\alpha=4$}

\label{s:beyond4} 

We conclude this paper with some observations regarding the general case when 
$\alpha>2$, conjectures, and a direction for future research.

We start with the following: 

\begin{conjecture}
\label{c:allalpha}
\cref{t:super4} holds true for every real number $\alpha>2$. 
\end{conjecture}

In \cref{r:Lcviab}, we have derived the Lorentz condition of $(c_{n,k})$ via $(b_{n,k})$ when $\alpha=4$.
On the other hand, by \cref{l:b4}, the Lorentz condition of $(b_{n,k})$ holds for every $\beta>1$.
Thus, if we can 
guarantee that $\max_{k\in \{0,1,\ldots,n\}}
|c_{n,k}-b_{n,k}|$ decays to $0$ sufficiently fast 
as $n\to\infty$, then the
Lorentz condition of $(c_{n,k})$ would follow immediately from that of
$(b_{n,k})$. Motivated by \cref{p:qrecursion}, we thus formulate another conjecture:

\begin{conjecture}
\label{c:close}
If $\beta = \alpha/2$, then 
$n \max_{k\in \{0,1,\ldots,n\}} |c_{n,k}-b_{n,k}| \to 0$ as $n\to\infty$. 
\end{conjecture}

While the coefficients $c_{n,k}$ share properties with $b_{n,k}$, 
as we have seen in the previous section, 
they are more complicated and do not have a closed-form formula as $b_{n,k}$ does.
Even unimodality, which is important for analyzing the Lorentz condition, 
turns out to be a much more subtle issue. 
To see this, we present the following formula which complements 
\cref{e:cnn1}:

\begin{proposition}
\label{p:jupp}
For a general $\alpha>2$, the coefficients $c_{n,k}$ from \cref{f:2.2} satisfy
\begin{equation}
\label{e:jupp1}
(\forall n\geq 1)\quad c_{n,1} = 
\frac{\alpha(\alpha n+2)}{(\alpha+2)(\alpha+2n)}c_{n,0}. 
\end{equation}
\end{proposition}
\begin{proof}
We abbreviate 
\begin{equation}
\label{e:jupp4}
A_n := \frac{2n+\alpha}{2(n+\alpha)} 
\quad\text{and}\quad
B_n := \frac{n}{n+\alpha},
\end{equation}
and prove the result by induction on $n\geq 1$.

Using \cref{e:cnk01}, we have
\begin{align}
c_{1,1} &= \frac{\alpha}{2(1+\alpha)}
= \frac{\alpha(\alpha+2)}{(\alpha+2)(\alpha+2)}\cdot \frac{2+\alpha}{2(1+\alpha)}
= \frac{\alpha(\alpha\cdot 1+2)}{(\alpha+2)(\alpha+2\cdot 1)}c_{1,0},
\end{align}
which verifies the base case. 

Now assume that \cref{e:jupp1} holds for some $n\geq 1$.
Using \cref{e:cnk} and the inductive hypothesis, we obtain
\begin{subequations}
\label{e:jupp2}
\begin{align}
c_{n+1,1}
&=A_{n+1}\bigl(c_{n,1}+c_{n,0}\bigr)-B_{n+1}c_{n-1,0}\\
&=
A_{n+1}c_{n,0}\left(1+\frac{\alpha(\alpha n+2)}{(\alpha+2)(\alpha+2n)}\right)-B_{n+1}c_{n-1,0}.
\end{align}
\end{subequations}
Using \cref{e:rcn0}, we also have 
$c_{n+1,0} = A_{n+1}c_{n,0}$ and $c_{n,0} = A_{n}c_{n-1,0}$;
thus, $c_{n-1,0} = c_{n+1,0}/(A_{n+1}A_n)$.
We now substitute these expressions into \cref{e:jupp2} and get 
\begin{subequations}
\begin{align}
c_{n+1,1} &= c_{n+1,0}\left(
1+\frac{\alpha(\alpha n+2)}{(\alpha+2)(\alpha+2n)}-\frac{B_{n+1}}{A_{n+1}A_n}
\right)\\
&= c_{n+1,0}\cdot \frac{\alpha(\alpha(n+1)+2)}{(\alpha+2)(\alpha+2(n+1))},
\label{e:jupp3b}
\end{align}
\end{subequations}
where \cref{e:jupp3b} follows from \cref{e:jupp4} and some algebraic manipulations.
\end{proof}

\begin{remark}[$\alpha = 1+\sqrt5$ is a critical parameter] 
Let $n\geq 1$. 
\cref{p:jupp} implies that 
\begin{equation}
c_{n,1}-c_{n,0}
=
c_{n,0}\,
\frac{n(\alpha^2-2\alpha-4)-\alpha^2}{(\alpha+2)(\alpha+2n)}.
\end{equation}
If $\alpha\leq 1+\sqrt5$, then
$\alpha^2-2\alpha-4\leq 0$, and hence 
$c_{n,1}-c_{n,0}<0$. 
If $1+\sqrt5 <\alpha$, then $c_{n,1}-c_{n,0} = 0$ provided that 
$n=\alpha^2/(\alpha^2-2\alpha-4)\in\{1,2,\ldots\}\subseteq\NN$. 
We obtain the following dichotomy: 

\hspace{1em}\makebox[1.2cm][l]{\rm\textbf{Either}}\quad
$\alpha\leq 1+\sqrt{5}$ and $c_{n,1}-c_{n,0}<0$;

\hspace{1em}\makebox[1.2cm][l]{\rm\textbf{Or}}\quad
$1+\sqrt5 <\alpha$, $N := \alpha^2/(\alpha^2-2\alpha-4)$, 
 and 
$c_{n,1}-c_{n,0} \begin{cases}
< 0, &\text{if $1\leq n<N$;}\\
=0, &\text{if $n=N\in\NN$;}\\
>0, &\text{if $N<n$.}
\end{cases}
$
\par 
\noindent
So for small $\alpha$  
or sufficiently small $n$, 
$(c_{n,k})_{k\in\{0,1,\ldots,n\}}$ is not unimodal!
Note that if $\alpha=4$, then $N = 4$ and hence 
$c_{n,0}>c_{n,1}$ for $n\leq 3$,
$c_{n,0}=c_{n,1}$ for $n=4$, and
$c_{n,0}<c_{n,1}$ for $n\geq 5$, 
which is consistent with \cref{p:unimod}.
\end{remark}






Motivated by numerical experiments, we formulate the following: 

\begin{conjecture}
No matter how $\alpha>2$ is chosen, 
the sequence $(c_{n,k})_{k\in\{1,2,\ldots,n\}}$ is unimodal. 
\end{conjecture}

The following remark was obtained with the help of \texttt{ChatGPT}:

\begin{remark}[why $\alpha=4$ is easier than other choices]
We have seen in \cref{e:260325a} that 
\begin{equation}
(\forall n\geq 1)\quad
c_{n,0} = \frac{2+\alpha}{2(1+\alpha)}\prod_{m=2}^n\frac{2m+\alpha}{2(m+\alpha)}.
\end{equation}
The remainder of this remark was pointed out to us by 
{\rm \texttt{ChatGPT}}: First, we can rewrite the product in
terms of the gamma function as 
\begin{align*}
(\forall n\geq 1)\quad
c_{n,0}
&= 
\frac{\alpha+2}{2(\alpha+1)}
\prod_{m=2}^n\frac{m+\alpha/2}{m+\alpha}\\
&= 
\frac{\alpha+2}{2(\alpha+1)}
\frac{\prod_{m=2}^n (m+\alpha/2)}{\prod_{m=2}^n (m+\alpha)}
\\
&= 
\frac{\alpha+2}{2(\alpha+1)}
\frac{\Gamma(n+1+\alpha/2)/\Gamma(2+\alpha/2)}{\Gamma(n+1+\alpha)/\Gamma(2+\alpha)}
\tag{by \cref{f:BG}\cref{f:BGnew}}\\
&= \frac{\alpha+2}{2(\alpha+1)} 
\cdot \frac{\Gamma(n+1+\alpha/2)}{\Gamma(n+1+\alpha)}\cdot \frac{\Gamma(2+\alpha)}{\Gamma(2+\alpha/2)}\\
&= \frac{1+\alpha/2}{\alpha+1}
\cdot \frac{\Gamma(n+1+\alpha/2)}{\Gamma(n+1+\alpha)}\cdot 
\frac{(\alpha+1)\Gamma(\alpha+1)}{(1+\alpha/2)\Gamma(1+\alpha/2)}
\tag{by \cref{f:BG}\cref{f:BG3}}
\\
&= \frac{\Gamma(n+1+\alpha/2)}{\Gamma(n+1+\alpha)}\cdot 
\frac{\Gamma(\alpha+1)}{\Gamma(1+\alpha/2)}.
\end{align*}
In view of \cref{f:BG}\cref{f:BG4}, it is clear that 
$c_{n,0}$ has a closed-form formula when $\alpha/2\in\NN$, i.e., 
when $\alpha$ is even. 
Using \cref{f:BG} and algebraic manipulations, one obtains for $n\geq 1$ 
the formulas 
\begin{align*}
\alpha=3:\quad &c_{n,0}=\frac{(2n+3)!}{4^n (n+1)!(n+3)!},\\[1mm]
\alpha=4:\quad &c_{n,0}=\frac{12}{(n+3)(n+4)},\\[1mm]
\alpha=5:\quad &c_{n,0}=\frac{(2n+6)!}{4^n (n+3)!(n+5)!},\\[1mm]
\alpha=6:\quad &c_{n,0}=\frac{120}{(n+4)(n+5)(n+6)},\\[1mm]
\alpha=7:\quad &c_{n,0}=\frac{3(2n+8)!}{4^n (n+4)!(n+7)!},\\[1mm]
\alpha=8:\quad &c_{n,0}=\frac{1680}{(n+5)(n+6)(n+7)(n+8)}.
\end{align*}
This shows that only even integer values for $\alpha$ will yield 
``pleasant'' closed-form formulas for $c_{n,0}$, with 
$\alpha=4$ being the simplest among them. 
While there is some hope to derive closed-form versions of $c_{n,k}$ 
at least for even integer values of $\alpha$, the complexity of 
$c_{n,0}$ suggests that the analysis is likely to be much more involved than
the analysis for $\alpha=4$ in \cref{s:alpha4}. 
\end{remark}

\begin{remark}[future work]
In this paper, we have analyzed a special case of the \BN\ iteration. 
In their original paper \cite{BN}, they study an even more general 
iteration, where $x_1$ need not be equal to $x_0$ and the actual iteration 
features another parameter and an additional difference of vectors. 
It would be interesting to see if the result in the present paper 
can be extended to their more general iteration.
\end{remark}

\section*{Acknowledgments}
The authors thank the editor and the reviewer
for constructive comments which significantly
improved the manuscript.
The research of HHB was partially supported by a Discovery Grant 
of the Natural Sciences and Engineering Research Council of
Canada. 

\section*{Data availability}
We do not analyze or generate any datasets, because our work
proceeds within a theoretical and mathematical approach. 
Relevant materials can be found in the list of references below.

\section*{Competing interests}
The authors have no competing interests to declare that are relevant to the content of this article.

\end{document}